\documentclass[a4paper,11pt,fleqn]{article}
\usepackage{enumitem}
\usepackage{amsmath}
\usepackage{amssymb}
\usepackage{pifont}
\usepackage[usenames,dvipsnames]{pstricks}
\usepackage{pstricks-add}
\usepackage{pst-grad}
\usepackage{pst-plot} 
\usepackage{pst-eucl} 
\usepackage{theorem} 
\usepackage{euscript}
\usepackage{exscale,relsize}
\usepackage{epic,eepic}
\usepackage{multicol}
\usepackage{tikz}
\usepackage{makeidx}
\usepackage{srcltx}     
\usepackage{url}
\usepackage{charter}
\usepackage{mathrsfs}    
\usepackage{graphicx}
\usepackage{authblk}
\newcommand{\email}[1]{\href{mailto:#1}{\nolinkurl{#1}}}
\usepackage[normalem]{ulem}   
\newlength{\mySubFigSize}
\definecolor{labelkey}{rgb}{0,0.08,0.45}
\definecolor{refkey}{rgb}{0,0.6,0.0}
\definecolor{Brown}{rgb}{0.45,0.0,0.05}
\definecolor{dgreen}{rgb}{0.00,0.49,0.00}
\definecolor{dblue}{rgb}{0,0.08,0.75}
\definecolor{lblue}{rgb}{0,0.7,0.75}
\RequirePackage[colorlinks,hyperindex]{hyperref}
\hypersetup{linktocpage=true,citecolor=dblue,linkcolor=dgreen}
\renewcommand{\leq}{\ensuremath{\leqslant}}
\renewcommand{\geq}{\ensuremath{\geqslant}}

\newcommand{\minimize}[2]{\ensuremath{\underset{\substack{{#1}}}%
{\text{\rm minimize}}\;\;#2 }}
 
\newcommand{\Scal}[2]{\bigg\langle{#1}\;\bigg|\:{#2}\bigg\rangle} 
\newcommand{\scal}[2]{{\langle{{#1}\mid{#2}}\rangle}}
\newcommand{\lai}[2]{{\ell_{#1}({#2)}}}
\newcommand{\yosi}[2]{\ensuremath{\sideset{^{#2}}{}{\operatorname{}}\!\!#1}}

\newcommand{\menge}[2]{\big\{{#1}~\big |~{#2}\big\}}

\newcommand{\HH}{\ensuremath{{\mathcal H}}}

\newcommand{\GG}{\ensuremath{{\mathcal G}}}
\newcommand{\KK}{\ensuremath{{\mathbb K}}}
\newcommand{\Sum}{\ensuremath{\displaystyle\sum}}

\newcommand{\emp}{\ensuremath{{\varnothing}}}

\newcommand{\Id}{\ensuremath{\operatorname{Id}}}

\newcommand{\RR}{\ensuremath{\mathbb{R}}}
\newcommand{\RP}{\ensuremath{\left[0,+\infty\right[}}

\newcommand{\RPP}{\ensuremath{\left]0,+\infty\right[}}

\newcommand{\RPX}{\ensuremath{\left[0,+\infty\right]}}
\newcommand{\RX}{\ensuremath{\left]-\infty,+\infty\right]}}

\newcommand{\NN}{\ensuremath{\mathbb N}}

\newcommand{\weakly}{\ensuremath{\:\rightharpoonup\:}}
\newcommand{\exi}{\ensuremath{\exists\,}}

\newcommand{\ran}{\ensuremath{\text{\rm ran}\,}}

\newcommand{\zer}{\ensuremath{\text{\rm zer}\,}}
\newcommand{\pinf}{\ensuremath{{+\infty}}}

\newcommand{\dom}{\ensuremath{\text{\rm dom}\,}}

\newcommand{\prox}{\ensuremath{\text{\rm prox}}}

\newcommand{\diam}{\ensuremath{\text{\rm diam}}}
\newcommand{\proj}{\ensuremath{\text{\rm proj}}}
\newcommand{\Fix}{\ensuremath{\text{\rm Fix}\,}}

\newcommand{\gra}{\ensuremath{\text{\rm gra}\,}}

\newcommand{\infconv}{\ensuremath{\mbox{\small$\,\square\,$}}}

\newcommand{\zeroun}{\ensuremath{\left]0,1\right[}}
\newcommand{\rzeroun}{\ensuremath{\left]0,1\right]}}

\newtheorem{theorem}{Theorem}[section]
\newtheorem{lemma}[theorem]{Lemma}
\newtheorem{corollary}[theorem]{Corollary}
\newtheorem{proposition}[theorem]{Proposition}
\newtheorem{assumption}[theorem]{Assumption}
\theoremstyle{plain}{\theorembodyfont{\rmfamily}%
\newtheorem{example}[theorem]{Example}}
\theoremstyle{plain}{\theorembodyfont{\rmfamily}%
\newtheorem{remark}[theorem]{Remark}}
\theoremstyle{plain}{\theorembodyfont{\rmfamily}%
}
\theoremstyle{plain}{\theorembodyfont{\rmfamily}%
}
\theoremstyle{plain}{\theorembodyfont{\rmfamily}%
}
\theoremstyle{plain}{\theorembodyfont{\rmfamily}%
\newtheorem{definition}[theorem]{Definition}}
\theoremstyle{plain}{\theorembodyfont{\rmfamily}%
\newtheorem{problem}[theorem]{Problem}}

\numberwithin{equation}{section}
\begin{document}

\title{\sffamily\huge\vskip -15mm 
Solving Composite Fixed Point Problems with\\
Block Updates\thanks{Contact 
author: P. L. Combettes, {\email{plc@math.ncsu.edu}},
phone:+1 (919) 515 2671. The work of P. L. Combettes was 
supported by the National Science Foundation under grant 
DMS-1818946 and that of L. E. Glaudin by ANR-3IA Artificial and 
Natural Intelligence Toulouse Institute.\\
\indent
{\bfseries ~2010 Mathematics Subject Classification:}
47J26, 47N10, 90C25, 47H05.}}

\author{Patrick L. Combettes$^1$ and Lilian E. Glaudin$^2$\\
\small $\!^1$North Carolina State University,
Department of Mathematics,
Raleigh, NC 27695-8205, USA\\
\small \email{plc@math.ncsu.edu}\\
\small \medskip
\small $\!^2$Universit\'e Toulouse Capitole, UMR TSE-R, 
31080 Toulouse, France\\
\small \email{lilian@glaudin.net}\\
}

\date{~}
\maketitle

\begin{abstract} 
\noindent
Various strategies are available to construct iteratively a common
fixed point of nonexpansive operators by activating only a block of
operators at each iteration. In the more challenging class of
composite fixed point problems involving operators that do not
share common fixed points, current methods require the activation
of all the operators at each iteration, and the question of
maintaining convergence while updating only blocks of operators is
open. We propose a method that achieves this goal and analyze its
asymptotic behavior. Weak, strong, and linear convergence results
are established by exploiting a connection with the theory of
concentrating arrays. Applications to several nonlinear and
nonsmooth analysis problems are presented, ranging from monotone
inclusions and inconsistent feasibility problems, to variational
inequalities and minimization problems arising in data science.
\end{abstract} 

\bigskip

\newpage
\section{Introduction}

Throughout, $\HH$ is a real Hilbert space with power set $2^{\HH}$,
identity operator $\Id$, scalar product $\scal{\cdot}{\cdot}$, and
associated norm $\|\cdot\|$. Recall that an operator
$T\colon\HH\to\HH$ is nonexpansive if it is $1$-Lipschitzian, and
$\alpha$-averaged for some $\alpha\in\zeroun$ if
$\Id+\alpha^{-1}(T-\Id)$ is nonexpansive \cite{Bail78}.  We
consider the broad class of nonlinear analysis problems which can
be cast in the following format.

\begin{problem}
\label{prob:1}
Let $m$ be a strictly positive integer and let
$(\omega_i)_{1\leq i\leq m}\in\rzeroun^m$ be such that
$\sum_{i=1}^m\omega_i=1$. For every
$i\in\{0,\ldots,m\}$, let $T_i\colon\HH\to\HH$ be
$\alpha_i$-averaged for some $\alpha_i\in\zeroun$.
The task is to find a fixed point of
$T_0\circ\sum_{i=1}^m\omega_iT_i$.
\end{problem}

A classical instantiation of Problem~\ref{prob:1} is found in the 
area of best approximation \cite{Che59a,Vonn49}: given two
nonempty closed convex subsets $C$ and $D$ of $\HH$, with 
projection operators $\proj_C$ and $\proj_D$, find a fixed point 
of the composition $\proj_C\circ\proj_D$. 
Geometrically, such points are those in $C$ at minimum distance 
from $D$, and they can be constructed via the method of alternating
projections \cite{Che59a,Gubi67}
\begin{equation}
\label{e:cheney59}
(\forall n\in\NN)\quad x_{n+1}=\proj_C(\proj_Dx_n).
\end{equation}
This problem was extended in \cite{Acke80} to that of finding a
fixed point of the composition $\prox_f\circ\prox_g$ of the
proximity operators of proper lower semicontinuous convex functions
$f\colon\HH\to\RX$ and $g\colon\HH\to\RX$. Recall that, given
$x\in\HH$, $\prox_fx$ is the unique minimizer of the function
$y\mapsto f(y)+\|x-y\|^2/2$ or, equivalently,
$\prox_fx=(\Id+\partial f)^{-1}$ where $\partial f$ is the
subdifferential of $f$, which is maximally monotone \cite{Livre1}.
A further generalization of this formalism was proposed in
\cite{Nonl05} where, given two maximally monotone operators
$A\colon\HH\to 2^{\HH}$ and $B\colon\HH\to 2^{\HH}$, with
associated resolvents $J_A=(\Id+A)^{-1}$ and $J_B=(\Id+B)^{-1}$,
the asymptotic behavior of the
iterations 
\begin{equation}
\label{e:nonl05}
(\forall n\in\NN)\quad x_{n+1}=J_A(J_Bx_n)
\end{equation}
for constructing a fixed point of $J_A\circ J_B$ was investigated.
We recall that $J_A$ and $J_B$ are $1/2$-averaged operators
\cite{Livre1}. Now let $A_0$ and $(B_i)_{1\leq i\leq m}$ be
maximally monotone operators from $\HH$ to $2^{\HH}$ and, for every
$i\in\{1,\ldots,m\}$, let
$\yosi{B_i}{\rho_i}=(\Id-J_{\rho_iB_i})/\rho_i$ be the Yosida
approximation of $B_i$ of index $\rho_i\in\RPP$. Set
$\beta=1/(\sum_{i=1}^m{1}/{\rho_i})$ and
$(\forall i\in\{1,\ldots,m\})$ $\omega_i={\beta}/{\rho_i}$.
In connection with the inclusion problem 
\begin{equation}
\label{e:11}
\text{find}\;\;x\in\HH\quad\text{such that}\quad
0\in A_0x+\sum_{i=1}^m\big(\yosi{B_i}{\rho_i}\big)x,
\end{equation}
the iterative process 
\begin{equation}
\label{e:main4}
(\forall n\in\NN)\quad
x_{n+1}=J_{\beta\gamma_n A_0}\Bigg(x_n+\displaystyle{\gamma_n
\Bigg(\sum_{i=1}^m\omega_i}J_{\rho_iB_i}x_n-x_n\Bigg)\Bigg),
\quad\text{where}\quad 0<\gamma_n<2,
\end{equation}
was studied in \cite{Opti04}. This algorithm captures 
\eqref{e:nonl05} as well as methods such as those proposed
in \cite{Lehd99,Leve09}; see also \cite{Wang11} for 
related problems. To make its structure more apparent, let us set
\begin{equation}
(\forall n\in\NN)\quad
T_{0,n}=J_{\beta\gamma_n A_0}
\quad\text{and}\quad 
(\forall i\in\{1,\ldots,m\})\quad 
T_{i,n}=(1-\gamma_n)\Id+\gamma_nJ_{\rho_iB_i}.
\end{equation}
Then we observe that, for every $n\in\NN$, the following hold:
\begin{itemize}
\setlength{\itemsep}{0pt}
\item
Problem \eqref{e:11} is the special case of Problem~\ref{prob:1} in
which $T_0=J_{\beta A_0}$, $T_1=J_{\rho_1B_1}$, \ldots, and
$T_m=J_{\rho_mB_m}$. Its set of solutions is 
\begin{equation}
\Fix\Bigg(J_{\beta A_0}\circ\sum_{i=1}^m\omega_iJ_{\rho_iB_i}\Bigg)=
\Fix\Bigg(T_{0,n}\circ\sum_{i=1}^m\omega_iT_{i,n}\Bigg).
\end{equation}
\item
For every $i\in\{0,\ldots,m\}$, $T_{i,n}$
is an averaged nonexpansive operator.
\item
The updating rule in \eqref{e:main4} can be written as 
\begin{equation}
\label{e:main5}
x_{n+1}=T_{0,n}\Bigg(\sum_{i=1}^m\omega_it_{i,n}\Bigg),
\quad\text{where}\quad(\forall i\in\{1,\ldots,m\})
\quad t_{i,n}=T_{i,n}x_n.
\end{equation}
\end{itemize}
The implementation of \eqref{e:main5} requires the activation of 
$T_{0,n}$ and 
the $m$ operators $(T_{i,n})_{1\leq i\leq m}$. If the operators
$(T_i)_{0\leq i\leq m}$ have common fixed points, then
Problem~\ref{prob:1} amounts to finding such a point, and this can
be achieved via block-iterative methods that 
require activating only subgroups of operators over the 
iterations; see, for instance, \cite{Aley08,Baus96,Else01,Flam95}.
In the absence of common fixed points, whether 
Problem~\ref{prob:1} can be solved by updating only subgroups of
operators is an open question. In the present paper, we address it
by showing that it is possible to lighten the
computational burden of iteration $n$ of \eqref{e:main5} by
activating only a subgroup 
$(T_{i,n})_{i\in I_n\subset\{1,\ldots,m\}}$ of 
the operators and by recycling older evaluations of the 
remaining operators. This leads to the iteration template
\begin{equation}
\label{e:1}
\begin{array}{l}
\left\lfloor
\begin{array}{l}
\text{for every}\;i\in I_n\\
\left\lfloor
\begin{array}{l}
t_{i,n}=T_{i,n}x_n
\end{array}
\right.\\
\text{for every}\;i\in\{1,\ldots,m\}\smallsetminus I_n\\
\left\lfloor
\begin{array}{l}
t_{i,n}=t_{i,n-1}\\
\end{array}
\right.\\
x_{n+1}=T_{0,n}\Bigg(\Sum_{i=1}^m\omega_i t_{i,n}\Bigg).
\end{array}
\right.\\
\end{array}
\end{equation}
The proposed framework will feature a flexible deterministic rule
for selecting the blocks of indices $(I_n)_{n\in\NN}$, as well as
tolerances in the evaluation of the operators in \eqref{e:1}.
Somewhat unexpectedly, our analysis will rely on the theory of
concentrating arrays, which appears predominantly in the area of
mean iteration methods
\cite{Siop17,Jmaa02,Kugl03,Mann53,Mann79,Reic75,Sadd10}. In
Section~\ref{sec:2}, we propose a new type of concentrating array
that will be employed in Section~\ref{sec:3} to investigate the
asymptotic behavior of the method. 
Finally, various applications to nonlinear analysis problems are
presented in Section~\ref{sec:4}.

\noindent
{\bfseries Notation.}
Let $M\colon\HH\to 2^{\HH}$. Then
$\gra M=\menge{(x,u)\in\HH\times\HH}{u\in Mx}$ is the graph
of $M$, 
$\zer M=\menge{x\in\HH}{0\in Mx}$ the set of zeros of $M$, 
$\dom M=\menge{x\in\HH}{Mx\neq\emp}$ the domain of $M$, 
$\ran M=\menge{u\in\HH}{(\exi x\in\HH)\;u\in Mx}$ the range of $M$,
$M^{-1}$ the inverse of $M$, which has graph
$\menge{(u,x)\in\HH\times\HH}{u\in Mx}$,
and $J_M=(\Id+M)^{-1}$ the resolvent of $M$. 
The parallel sum of $M$ and $A\colon\HH\to 2^{\HH}$ is
$M\infconv A=(M^{-1}+A^{-1})^{-1}$.
Further, $M$ is monotone if
\begin{equation}
\big(\forall (x,u)\in\gra M\big)
\big(\forall (y,v)\in\gra M\big)\quad\scal{x-y}{u-v}\geq 0,
\end{equation}
and maximally monotone if, in addition, there exists no
monotone operator $A\colon\HH\to 2^{\HH}$ such that
$\gra M\subset\gra A\neq\gra M$. If $M-\rho\Id$ is monotone for
some $\rho\in\RPP$, then $M$ is strongly monotone.
We denote by $\Gamma_0(\HH)$ the class of lower semicontinuous
convex functions $f\colon\HH\to\RX$ such that 
$\dom f=\menge{x\in\HH}{f(x)<\pinf}\neq\emp$.
Let $f\in\Gamma_0(\HH)$. The subdifferential of $f$
is the maximally monotone operator
$\partial f\colon\HH\to 2^{\HH}\colon
x\mapsto\menge{u\in\HH}{(\forall y\in\HH)\;
\scal{y-x}{u}+f(x)\leq f(y)}$.
For every $x\in\HH$, the unique minimizer of
the function $f+(1/2)\|\cdot-x\|^2$ is denoted by $\prox_fx$.
We have $\prox_f=J_{\partial f}$. Let $C$ be a nonempty closed
convex subset of $\HH$. Then $\proj_C$ is the projector onto $C$,
$d_C$ the distance function to $C$, and $\iota_C$ is the indicator
function of $C$, which takes the value $0$ on $C$ and $\pinf$ on
its complement.

\section{Concentrating arrays}
\label{sec:2}

Mann's mean value iteration method seeks a fixed point of an
operator $T\colon\HH\to\HH$ via the iterative process 
$x_{n+1}=T\overline{x}_n$, where $\overline{x}_n$ is a convex 
combination of the points $(x_j)_{0\leq j\leq n}$
\cite{Mann53,Mann79}.
The notion of a concentrating array was introduced in \cite{Jmaa02}
to study the asymptotic behavior of such methods. Interestingly, it
will turn out to be also quite useful in our investigation of the
asymptotic behavior of \eqref{e:1}.

\begin{definition}{\rm\cite[Definition~2.1]{Jmaa02}}
\label{d:con}
A triangular array $(\mu_{n,j})_{n\in\NN,0\leq j\leq n}$ in $\RP$
is concentrating if the following hold:
\begin{enumerate}[label={\rm[\alph*]}]
\setlength{\itemsep}{0pt}
\item 
\label{d:coni}
$(\forall n\in\NN)$ $\sum_{j=0}^n\mu_{n,j}=1$.
\item 
\label{d:conii}
$(\forall j\in\NN)$ $\lim_{n\to\pinf}\mu_{n,j}=0$.
\item
\label{d:coniii}
Every sequence $(\xi_n)_{n\in\NN}$ in $\RP$ that satisfies
\begin{equation}
\label{e:12}
\big(\forall n\in\NN\big)\quad
\xi_{n+1}\leq\Sum_{j=0}^n\mu_{n,j}\xi_j+\varepsilon_n,
\end{equation}
for some summable sequence $(\varepsilon_n)_{n\in\NN}$ in $\RP$,
converges.
\end{enumerate}
\end{definition}

We shall require the following convergence principle, which
extends that of quasi-Fej\'er monotonicity \cite{Else01}.

\begin{lemma}
\label{l:1}
Let $C$ be a nonempty subset of $\HH$, let $\phi\colon\RP\to\RP$ be
strictly increasing and such that $\lim_{t\to\pinf}\phi(t)=\pinf$,
let $(x_n)_{n\in\NN}$ be a sequence in $\HH$, let
$(\mu_{n,j})_{n\in\NN,0\leq j\leq n}$ be a concentrating array in
$\RP$, let $(\beta)_{n\in\NN}$ be a sequence in $\RP$, and let 
$(\varepsilon_n)_{n\in\NN}$ be a summable sequence in $\RP$ such
that 
\begin{equation}
\label{e:x0}
(\forall x\in C)(\forall n\in\NN)\quad
\phi(\|x_{n+1}-x\|)\leq\Sum_{j=0}^{n}\mu_{n,j}\phi(\|x_j-x\|)
-\beta_n+\varepsilon_n.
\end{equation}
Then the following hold:
\begin{enumerate}
\setlength{\itemsep}{0pt}
\item
\label{l:1i}
$(x_n)_{n\in\NN}$ is bounded.
\item
\label{l:1ii}
$\beta_n\to 0$.
\item
\label{l:1iii}
Suppose that every weak sequential cluster point of 
$(x_n)_{n\in\NN}$ belongs to $C$. Then $(x_n)_{n\in\NN}$ converges
weakly to a point in $C$.
\item
\label{l:1iv}
Suppose that $(x_n)_{n\in\NN}$ has a strong sequential cluster
point in $C$. Then $(x_n)_{n\in\NN}$ converges strongly to a 
point in $C$.
\end{enumerate}
\end{lemma}
\begin{proof}
Let $x\in C$. Let us first show that 
\begin{equation}
\label{e:x10}
(\|x_n-x\|)_{n\in\NN}\;\text{converges.}
\end{equation}
It follows from \eqref{e:x0} and Definition~\ref{d:con} that 
$(\phi(\|x_n-x\|))_{n\in\NN}$ converges, say 
$\phi(\|x_n-x\|)\to\lambda$. However, since 
$\lim_{t\to\pinf}\phi(t)=\pinf$, $(\|x_n-x\|)_{n\in\NN}$ is bounded 
and, to establish \eqref{e:x10}, it suffices to show that it does
not have two distinct cluster points. Suppose to the contrary that
there exist subsequences $(\|x_{k_n}-x\|)_{n\in\NN}$ and 
$(\|x_{l_n}-x\|)_{n\in\NN}$ such that $\|x_{k_n}-x\|\to\eta$ and 
$\|x_{l_n}-x\|\to\zeta>\eta$, and fix 
$\varepsilon\in\left]0,(\zeta-\eta)/2\right[$. Then, for $n$
sufficiently large, $\|x_{k_n}-x\|\leq\eta+\varepsilon<
\zeta-\varepsilon\leq\|x_{l_n}-x\|$ and, since $\phi$ is strictly
increasing, $\phi(\|x_{k_n}-x\|)\leq\phi(\eta+\varepsilon)<
\phi(\zeta-\varepsilon)\leq\phi(\|x_{l_n}-x\|)$. Taking the limit as
$n\to\pinf$ yields $\lambda\leq\phi(\eta+\varepsilon)<
\phi(\zeta-\varepsilon)\leq\lambda$, which is impossible.

\ref{l:1i} and \ref{l:1iv}: Clear in view of \eqref{e:x10}.

\ref{l:1ii}: 
As shown above, there exists $\lambda\in\RP$ such that 
$\phi(\|x_n-x\|)\to\lambda$. In turn, \cite[Theorem~3.5.4]{Knop56} 
implies that
$\sum_{j=0}^n\mu_{n,j}\phi(\|x_j-x\|)\to\lambda$. We thus derive
from \eqref{e:x0} that 
$0\leq\beta_n\leq\sum_{j=0}^n\mu_{n,j}\phi(\|x_j-x\|)-
\phi(\|x_{n+1}-x\|)+\varepsilon_n\to 0$.

\ref{l:1iii}: 
This follows from \eqref{e:x10} and \cite[Lemma~2.47]{Livre1}.
\end{proof}

Several examples of concentrating arrays are provided in
\cite{Jmaa02}. Here is a novel construction which is not only of
interest to mean iteration processes in fixed point theory 
\cite{Siop17,Jmaa02,Kugl03,Mann53,Mann79,Sadd10} but
will also play a pivotal role in establishing our main result, 
Theorem~\ref{t:1}. 

\begin{proposition}
\label{p:conc}
Let $K$ be a strictly positive integer and let 
$(\mu_{n,j})_{n\in\NN,0\leq j\leq n}$ be a
triangular array in $\RP$ such that the following hold:
\begin{enumerate}
\setlength{\itemsep}{0pt}
\item 
\label{p:conci}
$(\forall n\in\NN)$ $\sum_{j=0}^n\mu_{n,j}=1$.
\item 
\label{p:concii}
$(\forall n\in\NN)(\forall j\in\NN)$ $n-j\geq K$ $\Rightarrow$
$\mu_{n,j}=0$.
\item
\label{p:conciii}
$\inf_{n\in\NN}\mu_{n,n}>0$.
\end{enumerate}
Then $(\mu_{n,j})_{n\in\NN,0\leq j\leq n}$ is a concentrating
array.
\end{proposition}
\begin{proof}
Properties \ref{d:coni} and \ref{d:conii} in
Definition~\ref{d:con} clearly hold. To verify \ref{d:coniii}, 
let $(\xi_n)_{n\in\NN}$ be a sequence in $\RP$ and let
$(\varepsilon_n)_{n\in\NN}$ be a summable sequence in $\RP$ such
that 
\begin{equation}
\label{e:13}
\big(\forall n\in\NN\big)\quad
\xi_{n+1}\leq\Sum_{j=0}^n\mu_{n,j}\xi_j+\varepsilon_n.
\end{equation}
Then, in view of \ref{p:concii}, for every integer $n\geq K-1$, 
\begin{equation}
\label{e:genM}
\xi_{n+1}\leq\Sum_{k=0}^{K-1}\mu_{n,n-k}\xi_{n-k}+\varepsilon_n.
\end{equation}
Set $\mu=\inf_{n\in\NN}\mu_{n,n}$. If $\mu=1$, then \ref{p:conci} 
and \eqref{e:genM} imply that, for every integer $n\geq K-1$, 
\begin{equation}
\label{e:gen7}
0\leq\xi_{n+1}\leq\xi_n+\varepsilon_n,
\end{equation}
and the convergence of $(\xi_n)_{n\in\NN}$ therefore follows from
\cite[Lemma~5.31]{Livre1}. We
henceforth assume that $\mu<1$ and, without loss of generality,
that $K>1$. For every integer $n\geq K-1$, define 
$\widehat{\xi}_n=\max_{0\leq k\leq K-1}\xi_{n-k}$, and observe
that \ref{p:conci} and \eqref{e:genM} yield
$\xi_{n+1}\leq\widehat{\xi}_n+\varepsilon_n$.
Hence,
\begin{equation}
(\forall n\in\{K-1,K,\ldots\})\quad 
0\leq\widehat{\xi}_{n+1}\leq\widehat{\xi}_n+\varepsilon_n
\end{equation}
and we deduce from \cite[Lemma~5.31]{Livre1} that 
$(\widehat{\xi}_n)_{n\in\NN}$ converges to 
some number $\eta\in\RP$. Therefore, if 
$(\xi_n)_{n\in\NN}$ converges, then its limit is $\eta$ as well. 
Let us argue by contradiction by assuming that
$\xi_n\not\to\eta$.
Then there exists $\nu\in\RPP$ such that 
\begin{equation}
\label{e:3l}
(\forall N\in\NN)(\exi n_0\in\{N,N+1,\ldots\})\quad 
|\xi_{n_0}-\eta|>\nu.
\end{equation}
Set 
\begin{equation}
\label{e:8}
\delta=\min\bigg\{\dfrac{\mu^{K-1}}{1-\mu^{K-1}},1\bigg\}
\quad\text{and}\quad
\nu'=\dfrac{\delta\nu}{4}.
\end{equation}
Since $\widehat{\xi}_n\to\eta$ 
and $\sum_{n\in\NN}\varepsilon_n<\pinf$, let us fix an 
integer $N\geq K-1$ such that
\begin{equation}
\label{e:prime}
(\forall n\in\{N,N+1,\ldots\})\quad
\eta-\mu^{K-1}\frac{\nu}{4}\leq
\widehat{\xi}_n\leq\eta+\nu'\quad\text{and}\quad
\Sum_{j\geq n}\varepsilon_j\leq(1-\mu^{K-1})\nu'.
\end{equation}
Then
\begin{equation}
\label{e:7}
(\forall k\in\{1,2,\ldots\})(\forall n\in\{N,N+1,\ldots\})
\quad\sum_{j=1}^{k}\mu^{j-1}\varepsilon_{n+k-j}\leq
\sum_{j\geq n}\varepsilon_j\leq(1-\mu^{K-1})\nu',
\end{equation}
while \eqref{e:genM} and \ref{p:conci} imply that 
\begin{align}
\label{e:3M}
(\forall n\in\{N,N+1,\ldots\})\quad\xi_{n+1}&\leq
\mu_{n,n}\xi_{n}+
\Sum_{k=1}^{K-1}\mu_{n,n-k}\xi_{n-k}+\varepsilon_n
\nonumber\\
&\leq
\mu_{n,n}\xi_n+(1-\mu_{n,n})\widehat{\xi}_n+\varepsilon_n
\nonumber\\
&=\mu\xi_n+(1-\mu)\widehat{\xi}_n
+(\mu_{n,n}-\mu)(\xi_n-\widehat{\xi}_n)+\varepsilon_n
\nonumber\\
&\leq\mu\xi_n+(1-\mu)\widehat{\xi}_n+\varepsilon_n
\nonumber\\
&\leq\mu\xi_n+(1-\mu)(\eta+\nu')+\varepsilon_n.
\end{align}
It follows from \eqref{e:3l} that there
exists an integer $n_0\geq N$ such that 
$|\xi_{n_0}-\eta|>\nu$, i.e., 
\begin{equation}
\label{e:4j}
\xi_{n_0}>\eta+\nu\quad\text{or}\quad
0\leq\xi_{n_0}<\eta-\nu.
\end{equation}
Suppose that $\xi_{n_0}>\eta+\nu$. Then 
\eqref{e:8} and \eqref{e:prime} imply that
$\nu<\xi_{n_0}-\eta\leq\widehat{\xi}_{n_0}-
\eta\leq\nu'\leq\nu/4$,
which is impossible. Therefore, $0\leq\xi_{n_0}<\eta-\nu$
and it follows from \eqref{e:3M} that
\begin{equation}
\label{e:9gf38}
\xi_{n_0+1}\leq\mu(\eta-\nu)+(1-\mu)(\eta+\nu')
+\varepsilon_{n_0}=\eta+(1-\mu)\nu'-\mu\nu
+\varepsilon_{n_0}.
\end{equation}
Let us show by induction that, for every integer $k\geq 1$, 
\begin{equation}
\label{e:fm8}
\xi_{n_0+k}\leq\eta+(1-\mu^k)\nu'-\mu^k\nu+
\Sum_{j=1}^{k}\mu^{j-1}\varepsilon_{n_0+k-j}.
\end{equation}
In view of \eqref{e:9gf38}, this inequality holds for $k=1$. 
Now suppose that it holds for some integer $k\geq 1$. 
Then we deduce from \eqref{e:3M} and \eqref{e:fm8} that 
\begin{align}
\label{e:fm8+1}
\xi_{n_0+k+1}
&\leq\mu\xi_{n_0+k}+\varepsilon_{n_0+k}+(1-\mu)(\eta+\nu')
\nonumber\\
&\leq\mu\eta+\mu(1-\mu^k)\nu'-\mu^{k+1}\nu
+\Sum_{j=0}^{k}\mu^{j}\varepsilon_{n_0+k-j}
+(1-\mu)(\eta+\nu')
\nonumber\\
&=\eta+(1-\mu^{k+1})\nu'
-\mu^{k+1}\nu
+\Sum_{j=1}^{k+1}\mu^{j-1}\varepsilon_{n_0+k+1-j}, 
\end{align}
which completes the induction argument.
Since $\mu\in\zeroun$, we derive from
\eqref{e:fm8}, \eqref{e:7}, and \eqref{e:8} that 
\begin{align}
\label{e:fmy}
(\forall k\in\{1,\ldots,K-1\})\quad 
\xi_{n_0+k}
&\leq\eta+(1-\mu^{k})\nu'-\mu^{k}\nu+(1-\mu^{K-1})\nu'\nonumber\\
&\leq\eta+2(1-\mu^{K-1})\nu'-\mu^{K-1}\nu\nonumber\\
&=\eta+(1-\mu^{K-1})\frac{\delta\nu}{2}
-\mu^{K-1}\nu\nonumber\\
&\leq\eta-\mu^{K-1}\frac{\nu}{2}.
\end{align}
Therefore, by \eqref{e:prime},
\begin{equation}
\eta-\mu^{K-1}\frac{\nu}{4}
\leq\widehat{\xi}_{n_0+K-1}\leq\eta
-\mu^{K-1}\frac{\nu}{2}. 
\end{equation}
We thus reach a contradiction and conclude that 
$(\xi_n)_{n\in\NN}$ converges.
\end{proof}

We derive from Proposition~\ref{p:conc} a new instance of 
a concentrating array on which the main result of 
Section~\ref{sec:3} will hinge.

\begin{example}
\label{ex:6}
Let $I$ be a nonempty finite set, let $(\omega_i)_{i\in I}$ be a 
family in $\rzeroun$ such that $\sum_{i\in I}\omega_i=1$, let
$(I_n)_{n\in\NN}$ be a sequence of nonempty subsets of $I$, and let
$K$ be a strictly positive integer such that $(\forall n\in\NN)$ 
$\bigcup_{0\leq k\leq K-1}I_{n+k}=I$. Set 
\begin{equation}
\label{e:z14}
(\forall n\in\NN)(\forall j\in\{0,\ldots,n\})\quad
\mu_{n,j}=
\begin{cases}
1,&\text{if}\;\;n=j<K;\\
\Sum_{i\in I_j\smallsetminus\bigcup_{k=j+1}^nI_k}
\omega_i,&\text{if}\;\;0\leq n-K<j;\\
0,&\text{otherwise.}
\end{cases}
\end{equation}
Then the following hold:
\begin{enumerate}
\setlength{\itemsep}{0pt}
\item
\label{ex:6i}
$(\mu_{n,j})_{n\in\NN,0\leq j\leq n}$ is a concentrating array.
\item
\label{ex:6ii}
Let $\NN\ni n\geq K-1$, let $(\xi_j)_{0\leq j\leq n}$ be 
in $\RP$, and, for every $i\in I$, define 
$\lai{i}{n}=\max\menge{k\in\{n-K+1,\ldots,n\}}{i\in I_k}$. Then
\begin{equation}
\label{e:z10}
\Sum_{j=0}^{n}\mu_{n,j}\xi_j=
\Sum_{i\in I}\omega_i\xi_{\lai{i}{n}}.
\end{equation}
\end{enumerate}
\end{example}
\begin{proof}
Let $n\in\NN$. If $n\geq K-1$, 
we have $\bigcup_{0\leq k\leq K-1}I_{n-k}=I$ and therefore 
\begin{multline}
\label{e:ti}
I\;\text{is the union of the disjoint sets}\\
\Bigg(I_{n},I_{n-1}\smallsetminus I_n,
I_{n-2}\smallsetminus (I_n\cup I_{n-1}),\ldots,
I_{n-K+2}\smallsetminus\bigcup_{k=n-K+3}^{n}I_{k},
I_{n-K+1}\smallsetminus\bigcup_{k=n-K+2}^{n}I_{k}\Bigg).
\end{multline}

\ref{ex:6i}:
It is clear from \eqref{e:z14} that, for every
integer $j\in[0,n-K]$, $\mu_{n,j}=0$. 
In turn, we derive from \eqref{e:z14} and \eqref{e:ti} that
\begin{equation}
\begin{cases}
\Sum_{j=0}^n\mu_{n,j}=
\mu_{n,n}=1,&\text{if}\;\;n<K;\\
\Sum_{j=0}^n\mu_{n,j}=\Sum_{j=n-K+1}^n\mu_{n,j}
=\Sum_{j=n-K+1}^n\:\Sum_{i\in I_j\smallsetminus
\bigcup_{k=j+1}^{n}I_k}\omega_i
=\Sum_{i\in I}\omega_i=1,&\text{if}\;\;n\geq K.
\end{cases}
\end{equation}
Finally, $\inf_{n\in\NN}\mu_{n,n}=\inf_{n\in\NN}
\sum_{i\in I_n}\omega_i\geq\min_{i\in I}\omega_i>0$. 
All the properties of Proposition~\ref{p:conc} are 
therefore satisfied.

\ref{ex:6ii}:
We have 
\begin{equation}
\label{e:z16}
\big(\forall j\in\{n-K+1,\ldots,n\}\big)
\Bigg(\forall i\in I_j\smallsetminus
\bigcup_{k=j+1}^{n}I_{k}\Bigg)\quad\lai{i}{n}=j.
\end{equation}
Hence, in view of \eqref{e:z14},
\begin{equation}
\label{e:z17}
\big(\forall j\in\{n-K+1,\ldots,n\}\big)\quad 
\Sum_{i\in I_j\smallsetminus
\bigcup_{k=j+1}^{n}I_k}\omega_i\xi_{\lai{i}{n}}=
\Sum_{i\in I_j\smallsetminus
\bigcup_{k=j+1}^{n}I_k}\omega_i\xi_j=\mu_{n,j}\xi_j.
\end{equation}
Consequently, \eqref{e:ti} yields
\begin{equation}
\Sum_{j=0}^n\mu_{n,j}\xi_j=\Sum_{j=n-K+1}^n\;
\Sum_{i\in I_j\smallsetminus\bigcup_{k=j+1}^{n}I_k}
\omega_i\xi_{\lai{i}{n}}=\Sum_{i\in I}\omega_i\xi_{\lai{i}{n}},
\end{equation}
which concludes the proof.
\end{proof}

\section{Solving Problem~\ref{prob:1} with block updates}
\label{sec:3}

We formalize the ideas underlying \eqref{e:1} by proposing a
method in which variable subgroups of operators are updated over
the course of the iterations, and establish its convergence 
properties. At iteration $n$, the block of operators to be
updated is $(T_{i,n})_{i\in I_n}$. For added flexibility, an error
$e_{i,n}$ is tolerated in the application of the operator
$T_{i,n}$. We operate under the following assumption, where $m$ is
as in Problem~\ref{prob:1}.

\begin{assumption}
\label{a:12}
$K$ is a strictly positive integer and 
$(I_n)_{n\in\NN}$ is a sequence of nonempty subsets of
$\{1,\ldots,m\}$ such that 
\begin{equation}
(\forall n\in\NN)\quad\bigcup_{k=0}^{K-1}I_{n+k}=\{1,\ldots,m\}.
\end{equation}
For every integer $n\geq K-1$, define
\begin{equation}
\label{e:ui2}
(\forall i\in\{1,\ldots,m\})\quad\lai{i}{n}=
\max\menge{k\in\{n-K+1,\ldots,n\}}{i\in I_k}.
\end{equation}
The sequences $(e_{0,n})_{n\in\NN}$, 
$(e_{1,n})_{n\in\NN}$, \ldots, 
$(e_{m,n})_{n\in\NN}$ are in $\HH$ and satisfy 
\begin{equation}
\label{e:96}
\sum_{n\geq K-1}\|e_{0,n}\|<\pinf\quad\text{and}\quad
(\forall i\in\{1,\ldots,m\})\quad
\sum_{n\geq K-1}\|e_{i,\lai{i}{n}}\|<\pinf.
\end{equation}
\end{assumption}

\begin{theorem}
\label{t:1}
Consider the setting of Problem~\ref{prob:1} together with
Assumption~\ref{a:12}. Let $\varepsilon\in\zeroun$ and, 
for every $n\in\NN$ and every $i\in\{0\}\cup I_n$, let 
$\alpha_{i,n}\in\left]0,1/(1+\varepsilon)\right[$ and let 
$T_{i,n}\colon\HH\to\HH$ be $\alpha_{i,n}$-averaged.
Suppose that, for every integer $n\geq K-1$, 
\begin{equation}
\label{e:f}
\emp\neq\Fix\Bigg(T_{0}\circ\sum_{i=1}^m\omega_iT_i\Bigg)
\subset\Fix\Bigg(T_{0,n}\circ\sum_{i=1}^m\omega_i
T_{i,\lai{i}{n}}\Bigg).
\end{equation}
Let $x_0\in\HH$, let $(t_{i,-1})_{1\leq i\leq m}\in\HH^m$, 
and iterate
\begin{equation}
\label{e:a1}
\begin{array}{l}
\text{for}\;n=0,1,\ldots\\
\left\lfloor
\begin{array}{l}
\text{for every}\;i\in I_n\\
\left\lfloor
\begin{array}{l}
t_{i,n}=T_{i,n}x_n+e_{i,n}
\end{array}
\right.\\
\text{for every}\;i\in\{1,\ldots,m\}\smallsetminus I_n\\
\left\lfloor
\begin{array}{l}
t_{i,n}=t_{i,n-1}\\
\end{array}
\right.\\
x_{n+1}=T_{0,n}\Bigg(\Sum_{i=1}^m\omega_it_{i,n}\Bigg)+e_{0,n}.
\end{array}
\right.\\
\end{array}
\end{equation}
Let $x$ be a solution to Problem~\ref{prob:1}.
Then the following hold:
\begin{enumerate}
\setlength{\itemsep}{0pt}
\item 
\label{t:1i-}
$(x_n)_{n\in\NN}$ is bounded. 
\item 
\label{t:1i}
Let $i\in\{1,\ldots,m\}$. Then
$x_{\lai{i}{n}}-T_{i,\lai{i}{n}}x_{\lai{i}{n}}
+T_{i,\lai{i}{n}}x-x\to 0$.
\item 
\label{t:1ii}
Let $i\in\{1,\ldots,m\}$ and $j\in\{1,\ldots,m\}$. Then
$T_{i,\lai{i}{n}}x_{\lai{i}{n}}
-T_{j,\lai{j}{n}}x_{\lai{j}{n}}
-T_{i,\lai{i}{n}}x
+T_{j,\lai{j}{n}}x\to 0$.
\item 
\label{t:1iii-}
Let $i\in\{1,\ldots,m\}$. Then
$x_{\lai{i}{n}}-x_n\to 0$.
\item 
\label{t:1iii}
$x_n-T_{0,n}(\sum_{i=1}^m\omega_iT_{i,\lai{i}{n}}x_n)\to 0$.
\item 
\label{t:1iv}
Suppose that every weak sequential cluster point of 
$(x_n)_{n\in\NN}$ solves Problem~\ref{prob:1}. Then the following
hold:
\begin{enumerate}
\setlength{\itemsep}{0pt}
\item
\label{t:1iva}
$(x_n)_{n\in\NN}$ converges weakly to a solution to
Problem~\ref{prob:1}.
\item
\label{t:1ivb}
Suppose that $(x_n)_{n\in\NN}$ has a strong sequential cluster 
point. Then $(x_n)_{n\in\NN}$ converges strongly to a solution to
Problem~\ref{prob:1}.
\end{enumerate}
\item
\label{t:1vi}
For every $n\geq K-1$ and every $i\in\{0\}\cup I_n$,
let $\rho_{i}\in\rzeroun$ be a Lipschitz constant of
$T_{i,n}$. Suppose that \eqref{e:a1} is implemented without errors
and that, for some $i\in\{0,\ldots,m\}$, $\rho_i<1$. 
Then $(x_n)_{n\in\NN}$ converges linearly 
to the unique solution to Problem~\ref{prob:1}.
\end{enumerate}
\end{theorem}
\begin{proof}
Let us fix temporarily an integer $n\geq K-1$. 
We first observe that, by nonexpansiveness of the
operators $T_{0,n}$ and $(T_{i,\lai{i}{n}})_{1\leq i\leq m}$,
\begin{align}
\label{e:49}
&\hskip -6mm
\big(\forall(y,e_0,\ldots,e_m)\in\HH^{m+2}\big)
\quad\bigg\|T_{0,n}\bigg(\Sum_{i=1}^m\omega_i\big(
T_{i,\lai{i}{n}}x_{\lai{i}{n}}+e_i\big)\bigg)+e_0-
T_{0,n}\bigg(\Sum_{i=1}^m\omega_i
T_{i,\lai{i}{n}}y\bigg)\bigg\|\nonumber\\
&\hskip 45mm\leq\bigg\|\Sum_{i=1}^m\omega_i
T_{i,\lai{i}{n}}x_{\lai{i}{n}}-
\Sum_{i=1}^m\omega_iT_{i,\lai{i}{n}}y+\sum_{i=1}^m\omega_ie_i
\bigg\|+\|e_0\|\nonumber\\
&\hskip 45mm\leq\Sum_{i=1}^m\omega_i\big
\|T_{i,\lai{i}{n}}x_{\lai{i}{n}}-
T_{i,\lai{i}{n}}y\big\|+\|e_0\|+\sum_{i=1}^m\omega_i\|e_i\|
\nonumber\\
&\hskip 45mm\leq\Sum_{i=1}^m\omega_i\|x_{\lai{i}{n}}-y\|+\|e_0\|+
\sum_{i=1}^m\|e_i\|.
\end{align}
We also note that \eqref{e:ui2} and \eqref{e:a1} yield
\begin{equation}
\label{e:y}
(\forall i\in\{1,\ldots,m\})\quad
t_{i,n}=T_{i,\lai{i}{n}}x_{\lai{i}{n}}+e_{i,\lai{i}{n}}.
\end{equation}
It follows from \eqref{e:a1}, \eqref{e:y}, \eqref{e:f},
and \eqref{e:49} that
\begin{align}
\label{e:zib1}
\|x_{n+1}-x\|
&=\bigg\|T_{0,n}\bigg(\Sum_{i=1}^m\omega_i\big(T_{i,\lai{i}{n}}
x_{\lai{i}{n}}+e_{i,\lai{i}{n}}\big)\bigg)+e_{0,n}-
T_{0,n}\bigg(\Sum_{i=1}^m\omega_i 
T_{i,\lai{i}{n}}x\bigg)\bigg\| \nonumber\\
&\leq\Sum_{i=1}^m\omega_i\|x_{\lai{i}{n}}-x\|+\|e_{0,n}\|+
\Sum_{i=1}^{m}\|e_{i,\lai{i}{n}}\|.
\end{align}
Now define 
$(\mu_{k,j})_{k\in\NN,0\leq j\leq k}$ as in 
\eqref{e:z14}, with $I=\{1,\ldots,m\}$, and set 
$\varepsilon_n=\|e_{0,n}\|+\sum_{i=1}^{m}\|e_{i,\lai{i}{n}}\|$.
Then we derive from Example~\ref{ex:6}\ref{ex:6ii} that
\begin{equation}
\label{e:h1}
\Sum_{i=1}^m\omega_i\|x_{\lai{i}{n}}-x\|=
\Sum_{j=0}^{n}\mu_{n,j}\|x_j-x\|,
\end{equation}
and it follows from \eqref{e:zib1} and \eqref{e:96} that 
\begin{equation}
\label{e:x7}
\|x_{n+1}-x\|\leq\Sum_{j=0}^{n}\mu_{n,j}\|x_j-x\|+\varepsilon_n,
\quad\text{where}\quad\sum_{k\geq K-1}\varepsilon_k<\pinf.
\end{equation}
Hence, Lemma~\ref{l:1}\ref{l:1i} guarantees that 
\begin{equation}
\label{e:x1}
(x_k)_{k\in\NN}\;\text{is bounded}.
\end{equation}
Consequently, using \eqref{e:96} and \eqref{e:49}, we obtain 
\begin{equation}
\label{e:n}
\nu_0=\sup_{k\geq K-1}
\Bigg(2\bigg\|T_{0,k}\bigg(\Sum_{i=1}^m\omega_i
\big(T_{i,\lai{i}{k}}x_{\lai{i}{k}}+e_{i,\lai{i}{k}}\big)\bigg)-
T_{0,k}\bigg(\Sum_{i=1}^m\omega_iT_{i,\lai{i}{k}}x\bigg)\bigg\|
+\|e_{0,k}\|\Bigg)<\pinf
\end{equation}
and 
\begin{equation}
\nu=\sup_{k\geq K-1}\Bigg(\Sum_{i=1}^m\omega_i\|e_{i,\lai{i}{k}}\|
+2\bigg\|\Sum_{i=1}^m\omega_i\big(T_{i,\lai{i}{k}} x_{\lai{i}{k}}
-T_{i,\lai{i}{k}}x\big)\bigg\|\Bigg)<\pinf.
\end{equation}
In addition, for every $y\in\HH$ and every $z\in\HH$, it follows
from \cite[Proposition~4.35]{Livre1} that
\begin{align}
\label{e:av2}
\|T_{0,n}y-T_{0,n}z\|^2
&\leq\|y-z\|^2-
\displaystyle{\frac{1-\alpha_{0,n}}{\alpha_{0,n}}}
\|(\Id-T_{0,n})y-(\Id-T_{0,n})z\|^2\nonumber\\
&\leq\|y-z\|^2-\varepsilon\|(\Id-T_{0,n})y-(\Id-T_{0,n})z\|^2
\end{align}
and, likewise, that
\begin{equation}
\label{e:av22}
(\forall i\in\{1,\ldots,m\})\quad
\|T_{i,\lai{i}{n}}y-T_{i,\lai{i}{n}}z\|^2\leq
\|y-z\|^2-\varepsilon
\|(\Id-T_{i,\lai{i}{n}})y-(\Id-T_{i,\lai{i}{n}})z\|^2.
\end{equation}
Hence, we deduce from \eqref{e:a1}, \eqref{e:y}, 
\eqref{e:f}, and \cite[Lemma~2.14(ii)]{Livre1} that
\begin{align}
\label{e:h0}
\|x_{n+1}-x\|^2
&=\bigg\|T_{0,n}\bigg(\Sum_{i=1}^m\omega_i(T_{i,\lai{i}{n}}
x_{\lai{i}{n}}+e_{i,\lai{i}{n}})\bigg)-
T_{0,n}\bigg(\Sum_{i=1}^m\omega_i 
T_{i,\lai{i}{n}}x\bigg)+e_{0,n}\bigg\|^2\nonumber\\
&\leq\bigg\|T_{0,n}\bigg(\Sum_{i=1}^m\omega_i(T_{i,\lai{i}{n}}
x_{\lai{i}{n}}+e_{i,\lai{i}{n}})\bigg)-
T_{0,n}\bigg(\Sum_{i=1}^m\omega_i 
T_{i,\lai{i}{n}}x\bigg)\bigg\|^2
+\nu_0\|e_{0,n}\| \nonumber\\
&\leq\bigg\|\Sum_{i=1}^m\omega_i\big(T_{i,\lai{i}{n}}
x_{\lai{i}{n}}-T_{i,\lai{i}{n}}x\big)
\bigg\|^2\nonumber\\
&\quad\;-\varepsilon 
\bigg\|(\Id-T_{0,n})\bigg(\Sum_{i=1}^m\omega_i(T_{i,\lai{i}{n}}
x_{\lai{i}{n}}+e_{i,\lai{i}{n}})\bigg)-
(\Id-T_{0,n})\bigg(\Sum_{i=1}^m\omega_i T_{i,\lai{i}{n}}
x\bigg)\bigg\|^2\nonumber\\
&\quad\;+\nu_0\|e_{0,n}\|+
\nu\Sum_{i=1}^m\omega_i\|e_{i,\lai{i}{n}}\|
\nonumber\\
&\leq\Sum_{i=1}^m\omega_i\|T_{i,\lai{i}{n}}
x_{\lai{i}{n}}-T_{i,\lai{i}{n}}x\|^2\nonumber\\
&\quad\;-\frac{1}{2}\Sum_{i=1}^m\Sum_{j=1}^m\omega_i\omega_j
\big\|T_{i,\lai{i}{n}}x_{\lai{i}{n}}-T_{i,\lai{i}{n}}x
-T_{j,\lai{j}{n}}x_{\lai{j}{n}}+
T_{j,\lai{j}{n}}x\big\|^2\nonumber\\
&\quad\;
-\varepsilon\bigg\|(\Id-T_{0,n})\bigg(\Sum_{i=1}^m
\omega_i(T_{i,\lai{i}{n}}x_{\lai{i}{n}}+e_{i,\lai{i}{n}})\bigg)+x-
\Sum_{i=1}^m\omega_i T_{i,\lai{i}{n}}x\bigg\|^2
\nonumber\\
&\quad\; +\nu_0\|e_{0,n}\|
+\nu\Sum_{i=1}^m\omega_i\|e_{i,\lai{i}{n}}\|\nonumber\\
&\leq\Sum_{i=1}^m\omega_i\|x_{\lai{i}{n}}-x\|^2
-\varepsilon\Sum_{i=1}^m\omega_i
\big\|(\Id-T_{i,\lai{i}{n}})x_{\lai{i}{n}}
-(\Id-T_{i,\lai{i}{n}})x\big\|^2\nonumber\\
&\quad\;-\frac{1}{2}\Sum_{i=1}^m\Sum_{j=1}^m\omega_i\omega_j
\big\|T_{i,\lai{i}{n}}x_{\lai{i}{n}}-
T_{i,\lai{i}{n}}x-T_{j,\lai{j}{n}}x_{\lai{j}{n}}+
T_{j,\lai{j}{n}}x\big\|^2\nonumber\\
&\quad\;
-\varepsilon\bigg\|(\Id-T_{0,n})\bigg(\Sum_{i=1}^m
\omega_i(T_{i,\lai{i}{n}}x_{\lai{i}{n}}+e_{i,\lai{i}{n}})\bigg)+x-
\Sum_{i=1}^m\omega_i T_{i,\lai{i}{n}}x\bigg\|^2\nonumber\\
&\quad\; +\nu_0\|e_{0,n}\|
+\nu\Sum_{i=1}^m\omega_i\|e_{i,\lai{i}{n}}\|.
\end{align}
It therefore follows from \eqref{e:h1} that 
\begin{align}
\label{e:h3}
&\hskip -9mm 
\|x_{n+1}-x\|^2\nonumber\\
&\hskip -4mm\leq\Sum_{j=0}^{n}\mu_{n,j}\|x_j-x\|^2
-\varepsilon\Sum_{i=1}^m\omega_i
\big\|x_{\lai{i}{n}}-T_{i,\lai{i}{n}}x_{\lai{i}{n}}
+T_{i,\lai{i}{n}}x-x\big\|^2\nonumber\\
&\hskip -4mm\quad\;-\frac{1}{2}\Sum_{i=1}^m\Sum_{j=1}^m
\omega_i\omega_j\big\|T_{i,\lai{i}{n}}x_{\lai{i}{n}}-
T_{i,\lai{i}{n}}x-T_{j,\lai{j}{n}}x_{\lai{j}{n}}+
T_{j,\lai{j}{n}}x\big\|^2\nonumber\\
&\hskip -4mm\quad\;
-\varepsilon\bigg\|\Sum_{i=1}^m\omega_i(T_{i,\lai{i}{n}}
x_{\lai{i}{n}}+e_{i,\lai{i}{n}})-T_{0,n}\bigg(\Sum_{i=1}^m\omega_i
(T_{i,\lai{i}{n}}x_{\lai{i}{n}}+e_{i,\lai{i}{n}})\bigg)
+x-\Sum_{i=1}^m\omega_i T_{i,\lai{i}{n}}x\bigg\|^2\nonumber\\
&\hskip -4mm\quad\; +\nu_0\|e_{0,n}\|
+\nu\Sum_{i=1}^m\omega_i\|e_{i,\lai{i}{n}}\|.
\end{align}
Hence, Example~\ref{ex:6}\ref{ex:6i}, \eqref{e:96}, and 
Lemma~\ref{l:1}\ref{l:1ii} imply that 
\begin{equation}
\label{e:x2}
\begin{cases}
\displaystyle{\max_{\substack{1\leq i\leq m}}}
\big\|x_{\lai{i}{n}}-T_{i,\lai{i}{n}}x_{\lai{i}{n}}
+T_{i,\lai{i}{n}}x-x\big\|\to 0\\
\displaystyle{\max_{\substack{1\leq i\leq m\\ 1\leq j\leq m}}}
\big\|T_{i,\lai{i}{n}}x_{\lai{i}{n}}
-T_{j,\lai{j}{n}}x_{\lai{j}{n}}-
T_{i,\lai{i}{n}}x+T_{j,\lai{j}{n}}x\big\|\to 0,
\end{cases}
\end{equation}
and that
\begin{equation}
\label{e:t0_i}
\bigg\|\Sum_{i=1}^m
\omega_i(T_{i,\lai{i}{n}}x_{\lai{i}{n}}+e_{i,\lai{i}{n}})
-T_{0,n}\bigg(\Sum_{i=1}^m\omega_i
(T_{i,\lai{i}{n}}x_{\lai{i}{n}}+e_{i,\lai{i}{n}})\bigg)
+x-\Sum_{i=1}^m\omega_i T_{i,\lai{i}{n}}x\bigg\|\to 0.
\end{equation}

\ref{t:1i-}: See \eqref{e:x1}. 

\ref{t:1i}--\ref{t:1ii}: See \eqref{e:x2}.

\ref{t:1iii-}--\ref{t:1iii}: 
It follows from \ref{t:1i} that
\begin{equation}
\label{e:max12-0}
\Sum_{i=1}^m\omega_ix_{\lai{i}{n}}-
\Sum_{i=1}^m\omega_iT_{i,\lai{i}{n}}x_{\lai{i}{n}}+
\Sum_{i=1}^m\omega_iT_{i,\lai{i}{n}}x-x\to 0.
\end{equation}
We also derive from \ref{t:1i} that, for every $i$ and every $j$
in $\{1,\ldots,m\}$, 
\begin{equation}
\label{e:y6}
x_{\lai{i}{n}}-T_{i,\lai{i}{n}}x_{\lai{i}{n}}
-x_{\lai{j}{n}}+T_{j,\lai{j}{n}}x_{\lai{j}{n}}
+T_{i,\lai{i}{n}}x-T_{j,\lai{j}{n}}x\to 0.
\end{equation}
Combining \ref{t:1ii} and \eqref{e:y6}, we obtain
\begin{equation}
\label{e:max3-3}
(\forall i\in\{1,\ldots,m\})(\forall j\in\{1,\ldots,m\})\quad
x_{\lai{i}{n}}-x_{\lai{j}{n}}\to 0.
\end{equation}
Now, let $\bar{\imath}\in\{1,\ldots,m\}$ and $\delta\in\RPP$.
Then \eqref{e:max3-3} implies that, for every 
$j\in\{1,\ldots,m\}$, there exists an integer 
$\overline{N}_{\delta,j}\geq K-1$ such that 
\begin{equation}
\big(\forall n\in\big\{\overline{N}_{\delta,j},
\overline{N}_{\delta,j}+1,\ldots\big\}\big)\quad
\|x_\lai{{\bar\imath}}{n}-x_\lai{j}{n}\|\leq\delta.
\end{equation}
Set 
$\overline{N}_{\delta}=\max_{1\leq j\leq m}\overline{N}_{\delta,j}$.
Then 
\begin{equation}
\big(\forall j\in\{1,\ldots,m\}\big)
\big(\forall n\in\big\{\overline{N}_{\delta},
\overline{N}_{\delta}+1,\ldots\big\}\big)\quad
\|x_\lai{{\bar\imath}}{n}-x_\lai{j}{n}\|\leq\delta.
\end{equation}
Thus, in view of \eqref{e:ui2}, for every integer 
$n\geq\overline{N}_{\delta}$, taking
$j_n\in I_n$ yields $\lai{j_n}{n}=n$ and hence
$\|x_{\lai{\bar\imath}{n}}-x_n\|\leq\delta$. This shows that
\begin{equation}
\label{e:4-1}
(\forall i\in\{1,\ldots,m\})\quad x_{\lai{i}{n}}-x_n\to 0.
\end{equation}
Consequently, it follows from \eqref{e:49} that
\begin{align}
\label{e:4-7}
&\hskip -3mm
\bigg\|T_{0,n}\bigg(\Sum_{i=1}^m\big(\omega_iT_{i,\lai{i}{n}}
x_{\lai{i}{n}}+e_{i,\lai{i}{n}}\big)\bigg)-
T_{0,n}\bigg(\Sum_{i=1}^m\omega_i
T_{i,\lai{i}{n}}x_n\bigg)\bigg\|\nonumber\\
&\leq\Sum_{i=1}^m\omega_i\|x_{\lai{i}{n}}-x_n\|+
\Sum_{i=1}^m\|e_{i,\lai{i}{n}}\|\nonumber\\
&\to 0.
\end{align}
In turn, we derive from \eqref{e:t0_i}, \eqref{e:max12-0}, 
\eqref{e:4-1}, and \eqref{e:96} that
\begin{align}
\label{e:u7}
&\hskip -5mm x_n-T_{0,n}\bigg(\Sum_{i=1}^m\omega_i
T_{i,\lai{i}{n}}x_n\bigg)\nonumber\\
&=T_{0,n}\bigg(\Sum_{i=1}^m\omega_i
(T_{i,\lai{i}{n}}x_{\lai{i}{n}}+e_{i,\lai{i}{n}})\bigg)
-T_{0,n}\bigg(\Sum_{i=1}^m\omega_i
T_{i,\lai{i}{n}}x_n\bigg)\nonumber\\
&\quad\;+\Sum_{i=1}^m\omega_i(T_{i,\lai{i}{n}}x_{\lai{i}{n}}
+e_{i,\lai{i}{n}})
-T_{0,n}\bigg(\Sum_{i=1}^m\omega_i
(T_{i,\lai{i}{n}}x_{\lai{i}{n}}+e_{i,\lai{i}{n}})\bigg)
+x-\Sum_{i=1}^m\omega_iT_{i,\lai{i}{n}}x\nonumber\\
&\quad\;
+\Sum_{i=1}^m\omega_ix_{\lai{i}{n}}
-\Sum_{i=1}^m\omega_iT_{i,\lai{i}{n}}x_{\lai{i}{n}}
+\Sum_{i=1}^m\omega_iT_{i,\lai{i}{n}}x-x
+\Sum_{i=1}^m\omega_i(x_n-x_{\lai{i}{n}})
-\Sum_{i=1}^m\omega_ie_{i,\lai{i}{n}}\nonumber\\
&\to 0.
\end{align}

\ref{t:1iva}: This follows from \eqref{e:x7} and
Lemma~\ref{l:1}\ref{l:1iii}.

\ref{t:1ivb}: By \ref{t:1iva}, there exists a solution $z$ to 
Problem~\ref{prob:1} such that $x_n\weakly z$. Therefore, $z$ must
be the strong cluster point in question, say $x_{k_n}\to z$. 
In view of \eqref{e:x7} and Lemma~\ref{l:1}\ref{l:1iv}, we conclude
that $x_n\to z$. 

\ref{t:1vi}:
Set $\rho=\rho_{0}\sum_{i=1}^m\omega_i\rho_{i}$ and note that
$\rho\in\zeroun$.
For every integer $n\geq K-1$ and every 
$(y_i)_{1\leq i\leq m}\in\HH^m$, \eqref{e:f} yields
\begin{align}
\label{e:w87}
\bigg\|T_{0,n}\bigg(\Sum_{i=1}^m\omega_i T_{i,\lai{i}{n}}y_i\bigg)-
x\bigg\|
&=\bigg\|T_{0,n}\bigg(\Sum_{i=1}^m\omega_i T_{i,\lai{i}{n}}
y_i\bigg)-T_{0,n}\bigg(\Sum_{i=1}^m\omega_i 
T_{i,\lai{i}{n}}x\bigg)\bigg\|\nonumber\\
&\leq\rho_{0}\bigg\|\Sum_{i=1}^m\omega_i T_{i,\lai{i}{n}}
y_i-\Sum_{i=1}^m\omega_i 
T_{i,\lai{i}{n}}x\bigg\|\nonumber\\
&\leq\rho_{0}\Sum_{i=1}^m\omega_i\|T_{i,\lai{i}{n}}
y_i-T_{i,\lai{i}{n}}x\|\nonumber\\
&\leq\rho_{0}\Sum_{i=1}^m\omega_i\rho_{i}
\big\|y_i-x\big\|.
\end{align}
Now let $y\in\Fix(T_0\circ \sum_{i=1}^{m}\omega_iT_i)$. Since
\eqref{e:w87} implies that
\begin{equation}
\label{e:h69}
\|y-x\|
=\bigg\|T_{0,K-1}\bigg(\Sum_{i=1}^m\omega_i 
T_{i,\lai{i}{K-1}}y\bigg)-x\bigg\|\leq\rho\|y-x\|,
\end{equation}
we infer that $y=x$, which shows uniqueness. For every integer 
$n\geq K-1$, \eqref{e:w87} also yields
\begin{equation}
\label{e:h70}
\|x_{n+1}-x\|
=\bigg\|T_{0,n}\bigg(\Sum_{i=1}^m\omega_i T_{i,\lai{i}{n}}
x_{\lai{i}{n}}\bigg)-x\bigg\|
\leq\rho_{0}\Sum_{i=1}^m\omega_i\rho_{i}
\|x_{\lai{i}{n}}-x\|.
\end{equation}
Now set
\begin{equation}
\label{e:z141}
(\forall n\in\NN)\quad\xi_n=\|x_n-x\|.
\end{equation}
It follows from \eqref{e:h70} that
\begin{equation}
\label{e:ineg12}
(\forall n\in\{K-1,K,\ldots\})\quad
\xi_{n+1}\leq\rho_{0}\Sum_{i=1}^m\omega_i\rho_{i}\xi_{\lai{i}{n}}
\leq\rho\widehat{\xi}_n,\quad\text{where}\quad
\widehat{\xi}_n=\max_{1\leq i\leq m}\xi_{\lai{i}{n}}.
\end{equation}
Let us show that
\begin{equation}
\label{e:k8}
(\forall n\in\NN)\quad
\xi_{n}\leq\rho^{\frac{n-K+1}{K}}\widehat{\xi}_{K-1}.
\end{equation}
We proceed by strong induction. We have
\begin{equation}
\label{e:indu2}
(\forall k\in\{0,\ldots,K-1\})\quad 
\xi_{k}\leq\widehat{\xi}_{K-1}
\leq\rho^{\frac{k-K+1}{K}}\widehat{\xi}_{K-1}.
\end{equation}
Next, let $\NN\ni n\geq K-1$ and suppose that
\begin{equation}
\label{e:indu}
(\forall k\in\{0,\ldots,n\})\quad
\xi_{k}\leq\rho^{\frac{k-K+1}{K}}\widehat{\xi}_{K-1}.
\end{equation}
Since $\{\lai{i}{n}\}_{1\leq i\leq m}\subset\{n-K+1,\ldots,n\}$,
there exists $k_n\in\{n-K+1,\ldots,n\}$ such 
that $\widehat{\xi}_n=\xi_{k_n}$. 
Therefore, we derive from \eqref{e:ineg12} and \eqref{e:indu} that
\begin{equation}
\xi_{n+1}\leq\rho\widehat{\xi}_n=\rho\xi_{k_n}
\leq\rho\rho^{\frac{k_n-K+1}{K}}\widehat{\xi}_{K-1}
=\rho^{\frac{k_n+1}{K}}\widehat{\xi}_{K-1}
\leq\rho^{\frac{n-K+2}{K}}\widehat{\xi}_{K-1}.
\end{equation}
We have thus shown that
\begin{equation}
\label{e:X}
(\forall n\in\NN)\quad
\|x_n-x\|\leq\rho^{\frac{1-K}{K}}\widehat{\xi}_{K-1}
\big(\rho^{\frac{1}{K}}\big)^{n},
\end{equation}
which establishes the linear convergence of
$(x_n)_{n\in\NN}$ to $x$.
\end{proof}

\begin{remark}
In applications, the cardinality of $I_n$ may 
be small compared to $m$. In such scenarios, it is advantageous to
set $z_{-1}=\sum_{i=1}^m\omega_i t_{i,-1}$ and write \eqref{e:a1} as 
\begin{equation}
\label{e:a88}
\begin{array}{l}
\text{for}\;n=0,1,\ldots\\
\left\lfloor
\begin{array}{l}
y_n=z_{n-1}-\Sum_{i\in I_n}\omega_it_{i,n-1}\\
\text{for every}\;i\in I_n\\
\left\lfloor
\begin{array}{l}
t_{i,n}=T_{i,n}x_n+e_{i,n}
\end{array}
\right.\\
\text{for every}\;i\in\{1,\ldots,m\}\smallsetminus I_n\\
\left\lfloor
\begin{array}{l}
t_{i,n}=t_{i,n-1}\\
\end{array}
\right.\\
z_n=y_n+\Sum_{i\in I_n}\omega_it_{i,n}\\
x_{n+1}=T_{0,n}z_n+e_{0,n},
\end{array}
\right.\\
\end{array}
\end{equation}
which provides a more economical update equation.
\end{remark}

Next, we specialize our results to the autonomous case, wherein
the operators $(T_i)_{0\leq i\leq m}$ of Problem~\ref{prob:1} are
used directly.

\begin{corollary}
\label{c:1}
Consider the setting of Problem~\ref{prob:1} under
Assumption~\ref{a:12} and the assumption that it has a solution. 
Let $x_0\in\HH$, let $(t_{i,-1})_{1\leq i\leq m}\in\HH^m$, and
iterate
\begin{equation}
\label{e:a2}
\begin{array}{l}
\text{for}\;n=0,1,\ldots\\
\left\lfloor
\begin{array}{l}
\text{for every}\;i\in I_n\\
\left\lfloor
\begin{array}{l}
t_{i,n}=T_{i}x_n+e_{i,n}
\end{array}
\right.\\
\text{for every}\;i\in\{1,\ldots,m\}\smallsetminus I_n\\
\left\lfloor
\begin{array}{l}
t_{i,n}=t_{i,n-1}\\
\end{array}
\right.\\[1mm]
x_{n+1}=T_{0}\Bigg(\Sum_{i=1}^m\omega_it_{i,n}\Bigg)+e_{0,n}.
\end{array}
\right.\\
\end{array}
\end{equation}
Then the following hold:
\begin{enumerate}
\setlength{\itemsep}{0pt}
\item 
\label{c:1i}
Let $x$ be a solution to Problem~\ref{prob:1} and let
$i\in\{1,\ldots,m\}$. Then $x_n-T_ix_n\to x-T_ix$.
\item 
\label{c:1ii}
$(x_n)_{n\in\NN}$ converges weakly to a solution to
Problem~\ref{prob:1}.
\item 
\label{c:1iii}
Suppose that, for some $i\in\{0,\ldots,m\}$, $T_i$ is 
\emph{demicompact} \cite{Petr66}, i.e., every bounded sequence
$(y_n)_{n\in\NN}$
such that $(y_n-T_iy_n)_{n\in\NN}$ converges has a strong 
sequential cluster point. Then $(x_n)_{n\in\NN}$ converges 
strongly to a solution to Problem~\ref{prob:1}.
\item
\label{c:1iv}
Suppose that \eqref{e:a1} is implemented without errors and that, 
for some $i\in\{0,\ldots,m\}$, $T_i$ is a Banach contraction. 
Then $(x_n)_{n\in\NN}$ converges linearly to the unique solution to
Problem~\ref{prob:1}.
\end{enumerate}
\end{corollary}
\begin{proof}
We operate in the special case of Theorem~\ref{t:1} for which
$(\forall n\in\NN)(\forall i\in\{0\}\cup I_n)$ $T_{i,n}=T_i$. 
Set $T=T_0\circ(\sum_{i=1}^m\omega_iT_i)$. Then the set
of solutions to Problem~\ref{prob:1} is $\Fix T$ and 
$T$ is nonexpansive since the operators 
$(T_i)_{0\leq i\leq m}$ are likewise. In addition, we derive from
Theorem~\ref{t:1}\ref{t:1iii} that 
\begin{equation}
\label{e:pg76}
x_n-Tx_n\to 0. 
\end{equation}
Altogether, \cite[Corollary~4.28]{Livre1}
asserts that, if $z\in\HH$ is a weak sequential cluster point of
$(x_n)_{n\in\NN}$, then $z\in\Fix T$. Thus, 
\begin{equation}
\label{e:71}
\text{every weak sequential cluster point of}\;
(x_n)_{n\in\NN}\;\text{solves Problem~\ref{prob:1}}.
\end{equation}
Recall from 
Theorem~\ref{t:1}\ref{t:1i} that 
\begin{equation}
\label{e:pg78}
(\forall x\in\Fix T)(\forall i\in\{1,\ldots,m\})\quad 
x_{\lai{i}{n}}-T_ix_{\lai{i}{n}}\to x-T_ix
\end{equation}
and from Theorem~\ref{t:1}\ref{t:1iii-} that 
\begin{equation}
\label{e:pg77}
(\forall i\in\{1,\ldots,m\})\quad x_{\lai{i}{n}}-x_n\to 0.
\end{equation}

\ref{c:1i}: 
We derive from the nonexpansiveness of $T_i$, \eqref{e:pg78}, and
\eqref{e:pg77} that 
\begin{align}
\label{e:pg80}
\|(\Id-T_i)x_n-(\Id-T_i)x\|
&\leq
\|(\Id-T_i)x_n-(\Id-T_i)x_{\lai{i}{n}}\|
+\|(\Id-T_i)x_{\lai{i}{n}}-(\Id-T_i)x\|
\nonumber\\
&\leq
2\|x_n-x_{\lai{i}{n}}\|
+\|(\Id-T_i)x_{\lai{i}{n}}-(\Id-T_i)x\|\to 0
\nonumber\\
&\to 0.
\end{align}

\ref{c:1ii}: 
This is a consequence of \eqref{e:71} and 
Theorem~\ref{t:1}\ref{t:1iva}.

\ref{c:1iii}:
In view of \eqref{e:71} and Theorem~\ref{t:1}\ref{t:1ivb}, it is
enough to show that $(x_n)_{n\in\NN}$ has a strong sequential
cluster point. 
It follows from \ref{c:1ii} and \cite[Lemma~2.46]{Livre1} that 
$(x_n)_{n\in\NN}$ is bounded. Hence, if $1\leq i\leq m$, we infer 
from \ref{c:1i} and the demicompactness of $T_i$ that
$(x_n)_{n\in\NN}$ has a strong sequential cluster point.
Now suppose that $i=0$ and let $x\in\Fix T$.
Arguing as in \eqref{e:t0_i}, we obtain
\begin{equation}
\label{e:ti8}
(\Id-T_0)\bigg(\Sum_{i=1}^m\omega_iT_ix_{\lai{i}{n}}\bigg)=
\Sum_{i=1}^m\omega_iT_ix_{\lai{i}{n}}
-T_0\bigg(\Sum_{i=1}^m\omega_iT_ix_{\lai{i}{n}}\bigg)
\to\Sum_{i=1}^m\omega_iT_ix-x.
\end{equation}
However, we derive from the nonexpansiveness of the operators 
$(T_i)_{0\leq i\leq m}$ and \eqref{e:pg77} that 
\begin{align}
\label{e:m4}
\bigg\|(\Id-T_0)\bigg(\Sum_{i=1}^m\omega_iT_ix_{n}\bigg)
-(\Id-T_0)\bigg(\Sum_{i=1}^m\omega_iT_ix_{\lai{i}{n}}
\bigg)\bigg\|
&\leq 2\bigg\|\Sum_{i=1}^m\omega_iT_ix_{n}
-\Sum_{i=1}^m\omega_iT_ix_{\lai{i}{n}}
\bigg\|\nonumber\\
&\leq 2\Sum_{i=1}^m\omega_i\bigg
\|T_ix_n-T_ix_{\lai{i}{n}}\bigg\|\nonumber\\
&\leq 2\|x_n-x_{\lai{i}{n}}\|\nonumber\\
&\to 0.
\end{align}
Combining \eqref{e:ti8} and \eqref{e:m4} yields 
\begin{equation}
\label{e:m6}
(\Id-T_0)\bigg(\Sum_{i=1}^m\omega_iT_ix_{n}\bigg)
\to\Sum_{i=1}^m\omega_iT_ix-x.
\end{equation}
Therefore, by demicompactness of $T_0$, the bounded sequence
$(\sum_{i=1}^m\omega_iT_ix_n)_{n\in\NN}$ has a strong
sequential cluster point and so does 
$(Tx_n)_{n\in\NN}=(T_0(\sum_{i=1}^m\omega_iT_ix_n))_{n\in\NN}$ 
since $T_0$ is nonexpansive. Consequently, \eqref{e:pg76} entails
that $(x_n)_{n\in\NN}$ has a strong sequential cluster point. 

\ref{c:1iv}: 
This is a consequence of Theorem~\ref{t:1}\ref{t:1vi}.
\end{proof}

In connection with Corollary~\ref{c:1}\ref{c:1iii}, here are 
examples of demicompact operators.

\begin{example}
\label{ex:7}
Le $T\colon\HH\to\HH$ be a nonexpansive operator. Then $T$ is
demicompact if one of the following holds:
\begin{enumerate}
\setlength{\itemsep}{0pt}
\item
\label{ex:7i}
$\ran T$ is boundedly relatively compact (the intersection of its
closure with every closed ball in $\HH$ is compact).
\item
\label{ex:7ii}
$\ran T$ lies in a finite-dimensional subspace.
\item
\label{ex:7iii}
$T=J_A$, where $A\colon\HH\to 2^{\HH}$ is maximally monotone and 
one of the following is satisfied:
\begin{enumerate}
\item
\label{ex:7iiia}
$A$ is \emph{demiregular} \cite{Sico10}, i.e., for every sequence 
$(x_n,u_n)_{n\in\NN}$ in $\gra A$ and for every $(x,u)\in\gra A$, 
[$x_n\weakly x$ and $u_n\to u$] $\Rightarrow$ $x_n\to x$.
\item
\label{ex:7iiib}
$A$ is uniformly monotone, i.e., there exists an increasing
function $\phi\colon\RP\to\RPX$ vanishing only at $0$ such that 
$(\forall (x,u)\in\gra A)(\forall (y,v)\in\gra A)$
$\scal{x-y}{u-v}\geq\phi(\|x-y\|)$.
\item
\label{ex:7iiic}
$A=\partial f$, where $f\in\Gamma_0(\HH)$ is uniformly convex, 
i.e., there exists an increasing function $\phi\colon\RP\to\RPX$
vanishing only at $0$ such that 
\begin{multline}
\label{e:unifconv}
(\forall\alpha\in\zeroun)(\forall x\in\dom f)(\forall y\in\dom f)\\
f\big(\alpha x+(1-\alpha) y\big)+\alpha(1-\alpha)\phi(\|x-y\|)
\leq \alpha f(x)+(1-\alpha)f(y).
\end{multline}
\item
\label{ex:7iiie}
$A=\partial f$, where $f\in\Gamma_0(\HH)$ and the lower level sets
of $f$ are boundedly compact.
\item
$\dom A$ is boundedly relatively compact.
\item
\label{ex:7iiih}
$A\colon\HH\to\HH$ is single-valued with a single-valued continuous
inverse.
\end{enumerate}
\end{enumerate}
\end{example}
\begin{proof}
Let $(y_n)_{n\in\NN}$ be a bounded sequence in $\HH$ such that
$y_n-Ty_n\to u$, for some $u\in\HH$. Set $(\forall n\in\NN)$
$x_n=Ty_n$. 

\ref{ex:7i}:
By construction, $(x_n)_{n\in\NN}$ lies in $\ran T$ and it
is bounded since $(\forall n\in\NN)$
$\|x_n\|\leq\|Ty_n-Ty_0\|+\|Ty_0\|\leq\|y_n-y_0\|+\|Ty_0\|$. Thus,
$(x_n)_{n\in\NN}$ lies in a compact set and it therefore possesses
a strongly convergent subsequence, say $x_{k_n}\to x\in\HH$. In
turn $y_{k_n}=y_{k_n}-Ty_{k_n}+x_{k_n}\to u+x$.

\ref{ex:7ii}$\Rightarrow$\ref{ex:7i}: Clear.

\ref{ex:7iiia}: Set $(\forall n\in\NN)$ $u_n=y_n-x_n$. Then
$u_n\to u$. In addition, $(\forall n\in\NN)$ $(x_n,u_n)\in\gra A$.
On the other hand, since $(y_n)_{n\in\NN}$ is bounded, we can
extract from it a weakly convergent subsequence, say 
$y_{k_n}\weakly y$. Then $x_{k_n}=y_{k_n}-u_{k_n}\weakly y-u$
and $u_{k_n}\to u$. By demiregularity, we get
$x_{k_n}\to y-u$ and therefore $y_{k_n}=x_{k_n}+u_{k_n}\to y$. 

\ref{ex:7iiib}--\ref{ex:7iiih}: These are special cases of 
\ref{ex:7iiia} \cite[Proposition~2.4]{Sico10}.
\end{proof}

\section{Applications}
\label{sec:4}

We present several applications of Theorem~\ref{t:1} to classical
nonlinear analysis problems which will be seen to reduce to
instantiations of Problem~\ref{prob:1}. These range from common
fixed point and inconsistent feasibility problems to composite
monotone inclusion and minimization problems. In each scenario, the
main benefit of the proposed framework will lie in its ability to
achieve convergence while updating only subgroups of the pool of
operators involved.

\subsection{Finding common fixed point of firmly nonexpansive
operators}

Firmly nonexpansive operators are operators which are 
$1/2$-averaged \cite{Livre1,Goeb84}. This application concerns the
following ubiquitous fixed point problem
\cite{Baus96,Cras95,Dye92a,Flam95,Tsen92}.

\begin{problem}
\label{prob:4}
Let $m$ be a strictly positive integer and, for every 
$i\in\{1,\ldots,m\}$, let $T_i\colon\HH\to\HH$ be
firmly nonexpansive. The task is to find a point in
$\bigcap_{i=1}^m\Fix T_i$.
\end{problem}

\begin{corollary}
\label{c:4}
Consider the setting of Problem~\ref{prob:4} under
Assumption~\ref{a:12} and the assumption that 
$\bigcap_{i=1}^m\Fix T_i\neq\emp$.
Let $(\omega_i)_{1\leq i\leq m}\in\rzeroun^m$ be such
that $\sum_{i=1}^m\omega_i=1$. For every $n\in\NN$ 
and every $i\in I_n$, let $T_{i,n}\colon\HH\to\HH$ be a firmly
nonexpansive operator such that $\Fix T_i\subset\Fix T_{i,n}$. 
Let $x_0\in\HH$, let $(t_{i,-1})_{1\leq i\leq m}\in\HH^m$, 
and iterate
\begin{equation}
\label{e:a17}
\begin{array}{l}
\text{for}\;n=0,1,\ldots\\
\left\lfloor
\begin{array}{l}
\text{for every}\;i\in I_n\\
\left\lfloor
\begin{array}{l}
t_{i,n}=T_{i,n}x_n+e_{i,n}
\end{array}
\right.\\
\text{for every}\;i\in\{1,\ldots,m\}\smallsetminus I_n\\
\left\lfloor
\begin{array}{l}
t_{i,n}=t_{i,n-1}\\
\end{array}
\right.\\
x_{n+1}=\Sum_{i=1}^m\omega_it_{i,n}.
\end{array}
\right.\\
\end{array}
\end{equation}
Then the following hold:
\begin{enumerate}
\item
\label{c:4i-}
Let $i\in \{1,\ldots,m\}$. Then
$(T_{i,\lai{i}{n}}x_{\lai{i}{n}})_{n\in\NN}$ is bounded.
\item
\label{c:4i}
Suppose that, for every $z\in\HH$, every $i\in\{1.\ldots,m\}$, and
every strictly increasing sequence $(k_n)_{n\in\NN}$ of integers
greater than $K$,
\begin{equation}
\label{e:j-p}
\begin{cases}
x_{\lai{i}{k_n}}\weakly z\\
x_{\lai{i}{k_n}}-T_{i,\lai{i}{k_n}}
x_{\lai{i}{k_n}}\to 0
\end{cases}
\quad\Rightarrow\quad z\in\Fix T_i.
\end{equation}
Then $(x_n)_{n\in\NN}$ converges weakly to a solution to
Problem~\ref{prob:4}.
\item
\label{c:4ii}
Suppose that, for some $i\in\{1,\ldots,m\}$, 
$(T_{i,\lai{i}{n}}x_{\lai{i}{n}})_{n\in\NN}$ 
has a strong sequential cluster point. 
Then $(x_n)_{n\in\NN}$ converges strongly to a solution to
Problem~\ref{prob:4}.
\end{enumerate}
\end{corollary}
\begin{proof}
Set $T_0=\Id$ and $(\forall i\in\{1,\ldots,m\})$ $\alpha_i=1/2$.
In addition, set $(\forall n\in\NN)$ $T_{0,n}=\Id$.
By assumption, for every 
$i\in\{1,\ldots,m\}$ and every integer $n\geq K-1$, 
$\Fix T_i\subset\Fix T_{i,\lai{i}{n}}$. Therefore, it
follows from \cite[Proposition~4.47]{Livre1} that, for every
integer $n\geq K-1$, 
\begin{equation}
\label{e:7y}
\Fix\Bigg(T_0\circ\sum_{i=1}^m\omega_iT_i\Bigg)=
\bigcap_{i=1}^m\Fix T_i
\subset\bigcap_{i=1}^m\Fix T_{i,\lai{i}{n}}=
\Fix\Bigg(T_{0,n}\circ\sum_{i=1}^m\omega_i
T_{i,\lai{i}{n}}\Bigg).
\end{equation}
This shows that \eqref{e:f} holds, that Problem~\ref{prob:4} is a
special case of Problem~\ref{prob:1}, and that \eqref{e:a17} is a
special case of \eqref{e:a1}. Let us derive the claims from
Theorem~\ref{t:1}. First, let $x\in\bigcap_{i=1}^m\Fix T_i$. Then,
for every $i\in\{1,\ldots,m\}$ and every integer $n\geq K-1$,
$x\in\Fix T_i\subset\Fix T_{i,\lai{i}{n}}$. This allows us to
deduce from Theorem~\ref{t:1}\ref{t:1i} that
\begin{equation}
\label{e:j-p2}
(\forall i\in\{1,\ldots,m\})\quad
x_{\lai{i}{n}}-T_{i,\lai{i}{n}}x_{\lai{i}{n}}
=x_{\lai{i}{n}}-T_{i,\lai{i}{n}}x_{\lai{i}{n}}
+T_{i,\lai{i}{n}}x-x\to 0.
\end{equation}
We also recall from Theorem~\ref{t:1}\ref{t:1iii-} that
\begin{equation}
\label{e:j-p3}
(\forall i\in\{1,\ldots,m\})\quad x_{\lai{i}{n}}-x_{n}\to 0.
\end{equation}

\ref{c:4i-}: 
This follows from Theorem~\ref{t:1}\ref{t:1i-}, \eqref{e:j-p2}, 
and \eqref{e:j-p3}.

\ref{c:4i}: Let $i\in\{1,\ldots,m\}$ and let $z\in\HH$ be a weak 
sequential cluster point of $(x_n)_{n\in\NN}$, say 
$x_{k_n}\weakly z$. In view of Theorem~\ref{t:1}\ref{t:1iva}, it
is enough to show that $z\in\Fix T_i$. We derive from
\eqref{e:j-p2} that $x_{\lai{i}{k_n}}-
T_{i,\lai{i}{k_n}}x_{\lai{i}{k_n}}\to 0$. On the other
hand, \eqref{e:j-p3} yields
$x_{\lai{i}{k_n}}=(x_{\lai{i}{k_n}}-x_{k_n})
+x_{k_n}\weakly z$. Using \eqref{e:j-p}, we obtain
$z\in\Fix T_i$.

\ref{c:4ii}: 
Let $z\in\HH$ be a strong sequential cluster point of 
$(T_{i,\lai{i}{n}}x_{\lai{i}{n}})_{n\in\NN}$, 
say $T_{i,\lai{i}{k_n}}x_{\lai{i}{k_n}}\to z$. 
Then \eqref{e:j-p2} yields $x_{\lai{i}{k_n}}\to z$. In turn,
\eqref{e:j-p3} implies that $x_{k_n}\to z$ and the conclusion
follows from Theorem~\ref{t:1}\ref{t:1ivb}. 
\end{proof}

\begin{example}
\label{ex:81}
We revisit a problem investigated in \cite{Sico04}. 
Let $m$ be a strictly positive integer, let
$(\omega_i)_{1\leq i\leq m}\in\rzeroun^m$ be such that
$\sum_{i=1}^m\omega_i=1$, and, for every $i\in\{1,\ldots,m\}$, let
$\rho_i\in\RP$ and let $A_i\colon\HH\to 2^{\HH}$ be maximally 
$\rho_i$-cohypomonotone in the sense that $A_i^{-1}+\rho_i\Id$ 
is maximally monotone. The task is to
\begin{equation}
\label{e:tp1}
\text{find}\;\;x\in\HH\quad\text{such that}\quad
(\forall i\in\{1,\ldots,m\})\quad 0\in A_ix,
\end{equation}
under the assumption that such a point exists. 
Suppose that Assumption~\ref{a:12} is satisfied,
let $\varepsilon\in\zeroun$, let $x_0\in\HH$, 
let $(t_{i,-1})_{1\leq i\leq m}\in\HH^m$, and let
$(\forall n\in\NN)(\forall i\in I_n)$
$\gamma_{i,n}\in\left[\rho_i+\varepsilon,\pinf\right[$.
Iterate
\begin{equation}
\label{e:a19}
\begin{array}{l}
\text{for}\;n=0,1,\ldots\\
\left\lfloor
\begin{array}{l}
\text{for every}\;i\in I_n\\
\left\lfloor
\begin{array}{l}
t_{i,n}=x_n+\big(1-{\rho_i}/{\gamma_{i,n}}\big)\big(
J_{\gamma_{i,n}A_i}x_n+e_{i,n}-x_n\big)
\end{array}
\right.\\
\text{for every}\;i\in\{1,\ldots,m\}\smallsetminus I_n\\
\left\lfloor
\begin{array}{l}
t_{i,n}=t_{i,n-1}\\
\end{array}
\right.\\
x_{n+1}=\Sum_{i=1}^m\omega_it_{i,n}.
\end{array}
\right.\\
\end{array}
\end{equation}
Then the following hold:
\begin{enumerate}
\setlength{\itemsep}{0pt}
\item
\label{ex:72i}
$(x_n)_{n\in\NN}$ converges weakly to a solution to 
\eqref{e:tp1}. 
\item
\label{ex:72ii}
Suppose that, for some $i\in\{1,\ldots,m\}$, $\dom A_i$ is 
boundedly relatively compact. Then $(x_n)_{n\in\NN}$ converges
strongly to a solution to \eqref{e:tp1}. 
\end{enumerate}
\end{example}
\begin{proof}
Set
\begin{equation}
\label{e:tp2}
(\forall i\in\{1,\ldots,m\})\quad
\begin{cases}
T_i=\Id+\bigg(1-\dfrac{\rho_i}{\gamma_i}\bigg)\big(
J_{\gamma_iA_i}-\Id\big),\quad\text{where}\quad
\gamma_i\in\left]\rho_i,\pinf\right[\\[3mm]
M_i=(A_i^{-1}+\rho_i\Id)^{-1}.
\end{cases}
\end{equation}
Then it follows from \cite[Proposition~20.22]{Livre1} that the
operators $(M_i)_{1\leq i\leq m}$ are maximally monotone and
therefore from \cite[Lemma~2.4]{Sico04} and
\cite[Corollary~23.9]{Livre1} that
\begin{equation}
\label{e:tp12}
(\forall i\in\{1,\ldots,m\})\quad
T_i=J_{(\gamma_i-\rho_i)M_i}
\;\text{is firmly nonexpansive and}\;\Fix T_i=\zer M_i=\zer A_i,
\end{equation}
which makes \eqref{e:tp1} an instantiation of
Problem~\ref{prob:4}. Now set
\begin{equation}
\label{e:tp3}
(\forall n\in\NN)(\forall i\in I_n)\quad 
T_{i,n}=\Id+\bigg(1-\dfrac{\rho_i}{\gamma_{i,n}}\bigg)
\big(J_{\gamma_{i,n}A_i}-\Id\big)
\quad\text{and}\quad
e_{i,n}'=\bigg(1-\dfrac{\rho_i}{\gamma_{i,n}}\bigg)e_{i,n}.
\end{equation}
Then $(\forall i\in\{1,\ldots,m\})$ 
$\sum_{n\geq K-1}\|e_{i,\lai{i}{n}}'\|
\leq\sum_{n\geq K-1}\|e_{i,\lai{i}{n}}\|<\pinf$. In addition,
$(\forall n\in\NN)(\forall i\in I_n)$ 
$t_{i,n}=T_{i,n}x_n+e_{i,n}'$.
This places \eqref{e:a19} in the same operating conditions as
\eqref{e:a17}. We also derive from \cite[Lemma~2.4]{Sico04} that
\begin{equation}
\label{e:tp9}
(\forall n\in\NN)(\forall i\in I_n)\quad 
T_{i,n}=J_{(\gamma_{i,n}-\rho_i)M_i}
\;\text{is firmly nonexpansive and}
\;\Fix T_{i,n}=\zer M_i=\zer A_i.
\end{equation}
\ref{ex:72i}:
In view of Corollary~\ref{c:4}\ref{c:4i}, it suffices to 
check that condition \eqref{e:j-p} holds. Let us take
$z\in\HH$, $i\in\{1,\ldots,m\}$, and a strictly increasing 
sequence $(k_n)_{n\in\NN}$ of integers greater than $K$ such that
\begin{equation}
\label{e:tp5}
x_{\lai{i}{k_n}}\weakly z\quad\text{and}\quad
x_{\lai{i}{k_n}}-T_{i,\lai{i}{k_n}}x_{\lai{i}{k_n}}\to 0.
\end{equation}
Then we must show that $0\in A_iz$. Note that
\begin{equation}
\label{e:tp15}
T_{i,\lai{i}{k_n}}x_{\lai{i}{k_n}}\weakly z.
\end{equation}
Now set 
\begin{equation}
\label{e:tp7}
(\forall n\in\NN)\quad
u_{\lai{i}{k_n}}=(\gamma_{i,\lai{i}{k_n}}-\rho_i)^{-1}
\big(x_{\lai{i}{k_n}}-T_{i,\lai{i}{k_n}}x_{\lai{i}{k_n}}\big).
\end{equation}
Then 
\begin{equation}
\label{e:tp16}
\|u_{\lai{i}{k_n}}\|
=\dfrac{\|x_{\lai{i}{k_n}}-
T_{i,\lai{i}{k_n}}x_{\lai{i}{k_n}}\|}
{\gamma_{i,\lai{i}{k_n}}-\rho_i}
\leq\dfrac{\|x_{\lai{i}{k_n}}-
T_{i,\lai{i}{k_n}}x_{\lai{i}{k_n}}\|}
{\varepsilon}\to 0.
\end{equation}
On the other hand, we derive from \eqref{e:tp9} that 
$(\forall n\in\NN)$
$T_{i,\lai{i}{k_n}}=J_{(\gamma_{i,\lai{i}{k_n}}-\rho_i)M_i}$.
Therefore, \eqref{e:tp7} yields
\begin{equation}
\label{e:tp4}
(\forall n\in\NN)\quad \big(T_{i,\lai{i}{k_n}}x_{\lai{i}{k_n}},
u_{\lai{i}{k_n}}\big)\in\gra M_i.
\end{equation}
However, since $M_i$ is maximally monotone, $\gra M_i$ is 
sequentially closed in 
$\HH^{\operatorname{weak}}\times\HH^{\operatorname{strong}}$
\cite[Proposition~20.38(ii)]{Livre1}. Hence, \eqref{e:tp15}, 
\eqref{e:tp16}, and \eqref{e:tp4} imply that 
$z\in\zer M_i=\zer A_i$.

\ref{ex:72ii}: By \eqref{e:tp9}, for every $n\geq K-1$,
$T_{i,\lai{i}{n}}x_{\lai{i}{n}}\in\ran
T_{i,\lai{i}{n}}=\dom(\Id+(\gamma_{i,\lai{i}{n}}-\rho_i)M_i)
=\dom M_i$. However, Corollary~\ref{c:4}\ref{c:4i-} asserts that
$(T_{i,\lai{i}{n}}x_{\lai{i}{n}})_{n\in\NN}$ lies in a closed
ball. Altogether, it possesses a strong sequential cluster point
and the conclusion follows from Corollary~\ref{c:4}\ref{c:4ii}.
\end{proof}

\begin{remark}
\label{r:tp}
Suppose that, in Example~\ref{ex:81}, the operators 
$(A_i)_{1\leq i\leq m}$ are maximally monotone, i.e.,
$(\forall i\in\{1,\ldots,m\})$ $\rho_i=0$. Suppose that, in
addition, all the operators are used at each iteration, i.e.,
$(\forall n\in\NN)$ $I_n=\{1,\ldots,m\}$. Then the implementation
of \eqref{e:a19} with no errors reduces to the barycentric 
proximal method of \cite{Lehd99}.
\end{remark}

\begin{example}
\label{ex:82}
As shown in \cite{Eusi20}, many problems in data science and
harmonic analysis can be cast as follows. Let $m$ be a strictly 
positive integer and, for every $i\in\{1,\ldots,m\}$, let 
$R_i\colon\HH\to\HH$ be firmly
nonexpansive and let $r_i\in\HH$. The task is to 
\begin{equation}
\label{e:z1}
\text{find}\;\;x\in\HH\quad\text{such that}\quad
(\forall i\in\{1,\ldots,m\})\quad r_i=R_ix,
\end{equation}
under the assumption that such a point exists. 
Let $(\omega_i)_{1\leq i\leq m}\in\rzeroun^m$ be such that
$\sum_{i=1}^m\omega_i=1$, suppose that Assumption~\ref{a:12} is
satisfied, let $x_0\in\HH$, and let 
$(t_{i,-1})_{1\leq i\leq m}\in\HH^m$.
Iterate
\begin{equation}
\label{e:a18}
\begin{array}{l}
\text{for}\;n=0,1,\ldots\\
\left\lfloor
\begin{array}{l}
\text{for every}\;i\in I_n\\
\left\lfloor
\begin{array}{l}
t_{i,n}=r_i+x_n-R_ix_n+e_{i,n}
\end{array}
\right.\\
\text{for every}\;i\in\{1,\ldots,m\}\smallsetminus I_n\\
\left\lfloor
\begin{array}{l}
t_{i,n}=t_{i,n-1}\\
\end{array}
\right.\\
x_{n+1}=\Sum_{i=1}^m\omega_it_{i,n}.
\end{array}
\right.\\
\end{array}
\end{equation}
Then the following hold:
\begin{enumerate}
\setlength{\itemsep}{0pt}
\item
\label{ex:82i}
$(x_n)_{n\in\NN}$ converges weakly to a solution to 
\eqref{e:z1}. 
\item
\label{ex:82ii}
Suppose that, for some $i\in\{1,\ldots,m\}$, $\Id-R_i$ is
demicompact.
Then $(x_n)_{n\in\NN}$ converges strongly to a solution to 
\eqref{e:z1}. 
\end{enumerate}
\end{example}
\begin{proof}
Following \cite{Eusi20}, \eqref{e:z1} can be formulated as an
instance of Problem~\ref{prob:4}, by choosing
$(\forall i\in\{1,\ldots,m\})$ $T_i=r_i+\Id-R_i$. A
straightforward implementation of \eqref{e:a17} consists of
setting $(\forall n\in\NN)(\forall i\in I_n)$ 
$T_{i,n}=T_i$, which reduces \eqref{e:a17} to \eqref{e:a18}.

\ref{ex:82i}:
Since the operators $(T_i)_{1\leq i\leq m}$ are nonexpansive,
\cite[Theorem~4.27]{Livre1} asserts that the operators 
$(\Id-T_i)_{1\leq i\leq m}$ are demiclosed, which implies that
condition \eqref{e:j-p} holds. Thus, the claim follows from 
Corollary~\ref{c:4}\ref{c:4i}.

\ref{ex:82ii}: 
We deduce from \eqref{e:j-p2} that 
$x_{\lai{i}{n}}-T_ix_{\lai{i}{n}}\to 0$, and from 
\eqref{e:j-p3} and \ref{ex:82i} that
$(x_{\lai{i}{n}})_{n\in\NN}$ is bounded. Hence, 
since $T_i$ is demicompact,
$(x_{\lai{i}{n}})_{n\in\NN}$ has a strong sequential cluster 
point and so does $(T_ix_{\lai{i}{n}})_{n\in\NN}$. We conclude
with Corollary~\ref{c:4}\ref{c:4ii}.
\end{proof}

\begin{remark}
\label{r:i}
If \eqref{e:z1} has no solution, \eqref{e:a18} will produce a fixed
point of the operator
$\sum_{i=1}^m\omega_iT_i=\Id+\sum_{i=1}^m\omega_i(r_i-R_i)$,
provided one exists. As discussed in \cite{Eusi20}, this is a valid
relaxation of \eqref{e:z1}.
\end{remark}

\subsection{Forward-backward operator splitting}

We consider the following monotone inclusion problem.

\begin{problem}
\label{prob:2}
Let $m$ be a strictly positive integer and let
$(\omega_i)_{1\leq i\leq m}\in\rzeroun^m$ be such that
$\sum_{i=1}^m\omega_i=1$. Let $A_0\colon\HH\to 2^{\HH}$ be 
maximally monotone and, for every $i\in\{1,\ldots,m\}$, 
let $\beta_i\in\RPP$ and let $A_i\colon\HH\to\HH$ be 
$\beta_i$-cocoercive, i.e.,
\begin{equation}
(\forall x\in\HH)(\forall y\in\HH)\quad
\scal{x-y}{A_ix-A_iy}\geq\beta_i\|A_ix-A_iy\|^2.
\end{equation}
The task is to find $x\in\HH$ such that 
$0\in A_0x+\sum_{i=1}^m\omega_i A_ix$. 
\end{problem}

\begin{remark}
\label{r:vi}
In Problem~\ref{prob:2}, suppose that $A_0$ is the normal cone
operator of a nonempty closed convex set $C$, i.e.,
$A_0=\partial\iota_C$. Then the problem is to solve the
variational inequality
\begin{equation}
\text{find}\;\;x\in C\quad\text{such that}\quad
(\forall y\in\HH)\quad\Scal{x-y}{\sum_{i=1}^m\omega_i A_ix}\leq 0.
\end{equation}
\end{remark}

If $m=1$, a standard method for solving Problem~\ref{prob:2} is the
forward-backward splitting algorithm \cite{Opti04,Tsen91,Bang13}. 
We propose below a multi-operator version of it with block-updates.

\begin{proposition}
\label{p:2}
Consider the setting of Problem~\ref{prob:2} under
Assumption~\ref{a:12} and the assumption that it has
a solution. Let 
$\gamma\in\left]0,2\min_{1\leq i\leq m}\beta_i\right[$,
let $x_0\in\HH$, 
let $(t_{i,-1})_{1\leq i\leq m}\in\HH^m$, and iterate
\begin{equation}
\label{e:a3}
\begin{array}{l}
\text{for}\;n=0,1,\ldots\\
\left\lfloor
\begin{array}{l}
\text{for every}\;i\in I_n\\
\left\lfloor
\begin{array}{l}
t_{i,n}=x_n-\gamma(A_ix_n+e_{i,n})\\
\end{array}
\right.\\
\text{for every}\;i\in\{1,\ldots,m\}\smallsetminus I_n\\
\left\lfloor
\begin{array}{l}
t_{i,n}=t_{i,n-1}\\
\end{array}
\right.\\[1mm]
x_{n+1}=J_{\gamma A_0}\Bigg(\Sum_{i=1}^m\omega_it_{i,n}\Bigg)
+e_{0,n}.
\end{array}
\right.\\
\end{array}
\end{equation}
Then the following hold:
\begin{enumerate}
\setlength{\itemsep}{0pt}
\item 
\label{p:2i}
Let $x$ be a solution to Problem~\ref{prob:2} and let
$i\in\{1,\ldots,m\}$. Then $A_ix_n\to A_ix$.
\item 
\label{p:2ii}
$(x_n)_{n\in\NN}$ converges weakly to a solution to
Problem~\ref{prob:2}.
\item 
\label{p:2iii}
Suppose that, for some $i\in\{0,\ldots,m\}$, $A_i$ is 
demiregular. Then $(x_n)_{n\in\NN}$ converges strongly to 
a solution to Problem~\ref{prob:2}.
\item
\label{p:2iv}
Suppose that, for some $i\in\{0,\ldots,m\}$, $A_i$ is strongly
monotone. Then $(x_n)_{n\in\NN}$ converges linearly to the unique 
solution to Problem~\ref{prob:2}.
\end{enumerate}
\end{proposition}
\begin{proof}
We apply Corollary~\ref{c:1} with $T_0=J_{\gamma A_0}$ and
$(\forall i\in\{1,\ldots,m\})$ $T_i=\Id-\gamma A_i$. It follows 
from \cite[Proposition~4.39 and Corollary~23.9]{Livre1} that the
operators $(T_i)_{0\leq i\leq m}$ are averaged, and hence from
\cite[Proposition~26.1(iv)(a)]{Livre1} that 
Problem~\ref{prob:2} coincides with Problem~\ref{prob:1}.
In addition, \eqref{e:a3} is an instance of \eqref{e:a2}. 

\ref{p:2i}:
See Corollary~\ref{c:1}\ref{c:1i}.

\ref{p:2ii}:
See Corollary~\ref{c:1}\ref{c:1ii}.

\ref{p:2iii}:
This follows from Corollary~\ref{c:1}\ref{c:1iii}. Indeed, 
if $i=0$, the demicompactness of $T_i$ follows from 
Example~\ref{ex:7}\ref{ex:7iiia}. On the other hand, if $i\neq 0$, 
take a bounded sequence $(y_n)_{n\in\NN}$ in $\HH$ such that 
$(y_n-T_iy_n)_{n\in\NN}$ converges, say $y_n-T_iy_n\to u$.
Then $A_iy_n\to u/\gamma$. On the other hand, 
$(y_n)_{n\in\NN}$ has a weak sequential cluster point, say
$y_{k_n}\weakly y$. So by demiregularity of $A_i$, 
$y_{k_n}\to y$, which shows that $T_i$ is demicompact.

\ref{p:2iv}:
If $i=0$, we derive from \cite[Proposition~23.13]{Livre1} that
$T_0=J_{\gamma A_0}$ is a Banach contraction. If $i\neq 0$, 
as in the proof of \cite[Proposition~26.16]{Livre1}, we obtain that
$T_i=\Id-\gamma A_i$ is a Banach contraction.
The conclusion follows from Corollary~\ref{c:1}\ref{c:1iv}.
\end{proof}

\begin{example}
\label{ex:c}
Consider maximally operators $A_0\colon\HH\to 2^{\HH}$ and,
for every $i\in\{1,\ldots,m\}$, $B_i\colon\HH\to 2^{\HH}$. 
The associated common zero problem is \cite{Else01,Lehd99,Zasl12} 
\begin{equation}
\label{e:cras1995}
\text{find}\;\;{x}\in\HH\;\;\text{such that}\;\;
0\in A_0{x}\cap\bigcap_{i=1}^mB_i{x}.
\end{equation}
As shown in \cite{Siop13}, when \eqref{e:cras1995} has no 
solution, a suitable relaxation is 
\begin{equation}
\label{e:p8}
\text{find}\;\;x\in\HH\quad\text{such that}\quad
0\in A_0x+\sum_{i=1}^m\omega_i(B_i\infconv C_i)x
\end{equation}
where, for every $i\in\{1,\ldots,m\}$, $C_i\colon\HH\to 2^{\HH}$ is
such that $C_i^{-1}$ is at most single-valued and strictly 
monotone, with $C_i^{-1} 0=\{0\}$. In this setting, if
\eqref{e:cras1995} happens to have solutions, they coincide with
those of \eqref{e:p8} \cite{Siop13}. Let us consider the
particular instance in which, for every $i\in \{1,\ldots,m\}$, 
$C_i$ is cocoercive, and set $A_i=B_i\infconv C_i$. Then the 
operators $(C_i^{-1})_{1\leq i\leq m}$ are strongly monotone and,
therefore, the operators $(A_i)_{1\leq i\leq m}$ are cocoercive. In
addition, \eqref{e:p8} is a special case of Problem~\ref{prob:2},
which can be solved via Proposition~\ref{p:2}. Let us observe that
if we further specialize by setting, for every $i\in\{1,\ldots,m\}$,
$C_i=\rho_i^{-1}\Id$ for some $\rho_i\in\RPP$, then \eqref{e:p8}
reduces to \eqref{e:11}. 
\end{example}

We now focus on minimization problems.

\begin{problem}
\label{prob:3}
Let $m$ be a strictly positive integer and let
$(\omega_i)_{1\leq i\leq m}\in\rzeroun^m$ be such that
$\sum_{i=1}^m\omega_i=1$. Let $f_0\in\Gamma_0(\HH)$ 
and, for every $i\in\{1,\ldots,m\}$, 
let $\beta_i\in\RPP$ and let $f_i\colon\HH\to\RR$ be 
a differentiable convex function with a $1/\beta_i$-Lipschitzian
gradient. The task is to 
\begin{equation}
\label{e:prob3}
\minimize{x\in\HH}{f_0(x)+\sum_{i=1}^m\omega_i f_i(x)}. 
\end{equation}
\end{problem}

\begin{proposition}
\label{p:3}
Consider the setting of Problem~\ref{prob:3} under
Assumption~\ref{a:12} and assume that
\begin{equation}
\label{e:exi}
\lim_{\substack{x\in\HH\\ \|x\|\to\pinf}}
\bigg(f_0(x)+\sum_{i=1}^m\omega_i f_i(x)\bigg)=\pinf.
\end{equation}
Let $\gamma\in\left]0,2\min_{1\leq i\leq m}\beta_i\right[$,
let $x_0\in\HH$, let $(t_{i,-1})_{1\leq i\leq m}\in\HH^m$, and 
iterate
\begin{equation}
\label{e:a4}
\begin{array}{l}
\text{for}\;n=0,1,\ldots\\
\left\lfloor
\begin{array}{l}
\text{for every}\;i\in I_n\\
\left\lfloor
\begin{array}{l}
t_{i,n}=x_n-\gamma(\nabla f_i(x_n)+e_{i,n})\\
\end{array}
\right.\\
\text{for every}\;i\in\{1,\ldots,m\}\smallsetminus I_n\\
\left\lfloor
\begin{array}{l}
t_{i,n}=t_{i,n-1}\\
\end{array}
\right.\\[1mm]
x_{n+1}=\prox_{\gamma f_0}\Bigg(\Sum_{i=1}^m\omega_it_{i,n}\Bigg)
+e_{0,n}.
\end{array}
\right.\\
\end{array}
\end{equation}
Then the following hold:
\begin{enumerate}
\setlength{\itemsep}{0pt}
\item 
\label{p:3i}
Let $x$ be a solution to Problem~\ref{prob:3} and let
$i\in\{1,\ldots,m\}$. Then $\nabla f_i(x_n)\to\nabla f_i(x)$.
\item 
\label{p:3ii}
$(x_n)_{n\in\NN}$ converges weakly to a solution to
Problem~\ref{prob:3}. 
\item 
\label{p:3iii}
Suppose that, for some $i\in\{0,\ldots,m\}$, one of the following
holds:
\begin{enumerate}
\setlength{\itemsep}{0pt}
\item 
\label{p:3iiia}
$f_i$ is uniformly convex.
\item 
\label{p:3iiib}
The lower level sets of $f_i$ are boundedly compact.
\end{enumerate}
Then $(x_n)_{n\in\NN}$ converges strongly to a solution to
Problem~\ref{prob:3}.
\item
\label{p:3iv}
Suppose that, for some $i\in\{0,\ldots,m\}$, $f_i$ is strongly
convex, i.e., it satisfies \eqref{e:unifconv} with
$\phi=|\cdot|^2/2$. Then $(x_n)_{n\in\NN}$ converges linearly to
the unique solution to Problem~\ref{prob:3}.
\end{enumerate}
\end{proposition}
\begin{proof}
We derive from \cite[Theorem~20.25]{Livre1} that $A_0=\partial f_0$ 
is maximally monotone and from \cite[Corollary~18.17]{Livre1} that,
for every $i\in\{1,\ldots,m\}$, $A_i=\nabla f_i$ is
$\beta_i$-cocoercive. In this setting, it follows from 
\cite[Corollary~27.3(i)]{Livre1} that Problem~\ref{prob:2} reduces 
to Problem~\ref{prob:3}. On the other hand, since the assumptions 
imply that $f_0+\sum_{i=1}^m\omega_i f_i$ is proper, lower 
semicontinuous, convex, and coercive, it follows from 
\cite[Corollary~11.16(ii)]{Livre1} that Problem~\ref{prob:3} 
has a solution. The claims therefore follow from
Proposition~\ref{p:2}, 
Example~\ref{ex:7}\ref{ex:7iiic}\&\ref{ex:7iiie}, and
\cite[Example~22.4(iv)]{Livre1}.
\end{proof}
 
An algorithm related to \eqref{e:a4} has recently been proposed in 
\cite{Mish20} in a finite-dimensional setting; see also 
\cite{Mokh18} for a special case. 

We illustrate an application of Proposition~\ref{p:3} in the
context of a variational model that captures various formulations
found in data analysis. 

\begin{example}
\label{ex:23}
Suppose that $\HH$ is separable, let $(e_k)_{k\in\KK\subset\NN}$
be an orthonormal basis of $\HH$, and, for every $k\in\KK$, let 
$\psi_k\in\Gamma_0(\RR)$ be such that $\psi_k\geq 0=\psi_k(0)$. For
every $i\in\{1,\ldots,m\}$, let $0\neq a_i\in\HH$, let 
$\mu_i\in\RPP$, and let
$\phi_i\colon\RR\to\RP$ be a differentiable convex function such
that $\phi_i'$ is $\mu_i$-Lipschitzian. The task is to 
\begin{equation}
\label{e:prob9}
\minimize{x\in\HH}{\sum_{k\in\KK}\psi_k(\scal{x}{e_k})+\dfrac{1}{m}
\sum_{i=1}^m\phi_i(\scal{x}{a_i})}. 
\end{equation}
Let us note that \eqref{e:prob9} is an instantiation of
\eqref{e:prob3} with 
$f_0=\sum_{k\in\KK}\psi_k\circ\scal{\cdot}{e_k}$
and, for every $i\in\{1,\ldots,m\}$, 
$f_i=\phi_i\circ\scal{\cdot}{a_i}$ and
$\omega_i=1/m$. The fact that $f_0\in\Gamma_0(\HH)$ is established
in \cite{Smms05}, where it is also shown that, given
$\gamma\in\RPP$,
\begin{equation}
\label{e:prox}
\prox_{\gamma f_0}\colon x\mapsto\sum_{k\in\KK}\big(
\prox_{\gamma\psi_k}\scal{x}{e_k}\big)e_k.
\end{equation}
On the other hand, for every $i\in\{1,\ldots,m\}$,
$f_i$ is a differentiable convex function and its gradient 
\begin{equation}
\label{e:g9}
\nabla f_i\colon x\mapsto\phi_i'(\scal{x}{a_i})a_i
\end{equation}
has Lipschitz constant $\mu_i\|a_i\|^2$. Let 
$\gamma\in\left]0,2/(\max_{1\leq i\leq m}\mu_i\|a_i\|^2)\right[$ 
and let $(I_n)_{n\in\NN}$ be as in Assumption~\ref{a:12}.
In view of \eqref{e:a4}, \eqref{e:prox}, and \eqref{e:g9}, 
we can solve \eqref{e:prob9} via the algorithm
\begin{equation}
\label{e:a41}
\begin{array}{l}
\text{for}\;n=0,1,\ldots\\
\left\lfloor
\begin{array}{l}
\text{for every}\;i\in I_n\\
\left\lfloor
\begin{array}{l}
t_{i,n}=x_n-\gamma\phi_i'(\scal{x_n}{a_i})a_i\\
\end{array}
\right.\\
\text{for every}\;i\in\{1,\ldots,m\}\smallsetminus I_n\\
\left\lfloor
\begin{array}{l}
t_{i,n}=t_{i,n-1}\\
\end{array}
\right.\\[1mm]
y_n=\Sum_{i=1}^m\omega_it_{i,n}\\
x_{n+1}=\Sum_{k\in\KK}\big(
\prox_{\gamma\psi_k}\scal{y_n}{e_k}\big)e_k.
\end{array}
\right.\\
\end{array}
\end{equation}
Infinite-dimensional
instances of \eqref{e:prob9} are discussed in 
\cite{Save18,Smms05,Daub04,Demo09}.
A popular finite-dimensional setting is obtained by
choosing $\HH=\RR^N$, $\KK=\{1,\ldots,N\}$, $(e_k)_{1\leq k\leq N}$
as the canonical basis, $\alpha\in\RPP$, and, 
for every $k\in\KK$, 
$\psi_k=\alpha|\cdot|$. This reduces \eqref{e:prob9} to
\begin{equation}
\label{e:prob91}
\minimize{x\in\RR^N}{\alpha\|x\|_1+\sum_{i=1}^m
\phi_i(\scal{x}{a_i})}. 
\end{equation}
Thus, choosing for every $i\in\{1,\ldots,m\}$
$\phi_i\colon t\mapsto|t-\eta_i|^2$, 
where $\eta_i\in\RR$ models an observation, yields the Lasso 
formulation, whereas choosing 
$\phi_i\colon t\mapsto\ln(1+\exp(t))-\eta_it$, where
$\eta_i\in\{0,1\}$ models a label, yields the penalized 
logistic regression framework \cite{Hast09}. 
\end{example}

\subsection{Hard constrained inconsistent convex feasibility 
problems}

The next application revisits a model proposed in 
\cite{Siim19} to relax inconsistent feasibility problems.

\begin{problem}
\label{prob:8}
Let $m$ be a strictly positive integer and let
$(\omega_i)_{1\leq i\leq m}\in\rzeroun^m$ be such that
$\sum_{i=1}^m\omega_i=1$. Let $C_0$ be a nonempty closed convex 
subset of $\HH$ and, for every $i\in\{1,\ldots,m\}$, let $\GG_i$ be
a real Hilbert space, let
$L_i\colon\HH\to\GG_i$ be a nonzero bounded linear operator, 
let $D_i$ be a nonempty closed convex subset of $\GG_i$, let
$\mu_i\in\RPP$, and let $\phi_i\colon\RR\to\RP$ be an even
differentiable convex function that vanishes only at $0$ and
such that $\phi_i'$ is $\mu_i$-Lipschitzian.
The task is to 
\begin{equation}
\label{e:prob8}
\minimize{x\in C_0}{\sum_{i=1}^m\omega_i\phi_i
\big(d_{D_i}(L_ix)\big)}. 
\end{equation}
\end{problem}

The variational formulation \eqref{e:prob8} is a relaxation of the
convex feasibility problem 
\begin{equation}
\label{e:feas}
\text{find}\;\;x\in C_0\quad\text{such that}\quad
(\forall i\in\{1,\ldots,m\})\quad L_ix\in D_i
\end{equation}
in the sense that, if 
\eqref{e:feas} is consistent, then its solution set is precisely
that of \eqref{e:prob8}; see \cite[Section~4.4]{Siim19} for details
on this formulation and background on inconsistent convex
feasibility problems. Here $C_0$ models a hard constraint.
An early instance of \eqref{e:feas} as a 
relaxation of \eqref{e:prob8} is Legendre's method of least-squares
to deal with an inconsistent system of $m$ linear equations in
$\HH=\RR^N$ \cite{Lege05}. There, $C_0=\RR^N$ and, for
every $i\in\{1,\ldots,m\}$, $\GG_i=\RR$, $D_i=\{\beta_i\}$,
$L_i=\scal{\cdot}{a_i}$ for some $a_i\in\RR^N$ such that
$\|a_i\|=1$, $\omega_i=1/m$, and $\phi_i=|\cdot|^2$. The
formulation \eqref{e:prob8} can also be regarded as a smooth
version of the set-theoretic Fermat-Weber problem \cite{Mord12}
arising in location theory, namely,
\begin{equation}
\minimize{x\in\HH}{\dfrac{1}{m}\sum_{i=1}^md_{C_i}(x)}.
\end{equation}

The following version of the Closed Range Theorem 
will be required.

\begin{lemma}{\rm\cite[Theorem~8.18]{Deut01}}
\label{l:mi7}
Let $\GG$ be a real Hilbert space and let $L\colon\HH\to\GG$ be a
nonzero bounded linear operator. Then 
$\ran L$ is closed $\Leftrightarrow$ $\ran L^*\circ L$ is closed 
$\Leftrightarrow$ $(\exi\rho\in\RPP)(\forall x\in(\ker L)^\bot)$ 
$\|Lx\|\geq\rho\|x\|$.
\end{lemma}

\begin{corollary}
\label{c:8}
Consider the setting of Problem~\ref{prob:8} under one of the
following assumptions:
\begin{enumerate}[label={\rm[\alph*]}]
\setlength{\itemsep}{0pt}
\item
\label{c:8a}
There exists $j\in\{1,\ldots,m\}$ such that 
$\lim_{\|x\|\to\pinf}\big(\iota_{C_0}(x)+\phi_j
\big(d_{D_j}(L_jx)\big)\big)=\pinf$.
\item
\label{c:8b}
There exists $j\in\{1,\ldots,m\}$ such that
$\ran L_j$ is closed, $C_0\subset(\ker L_j)^\bot$, and $D_j$ is
bounded.
\item
\label{c:8c}
There exists $j\in\{1,\ldots,m\}$ such that
$\GG_j=\HH$, $L_j=\Id$, and $D_j$ is bounded.
\item
\label{c:8d}
$C_0$ is bounded. 
\end{enumerate}
Set $\beta=1/(\max_{1\leq i\leq m}\mu_i\|L_i\|^2)$, let 
$\gamma\in\left]0,2\beta\right[$, let
$(I_n)_{n\in\NN}$ be as in Assumption~\ref{a:12},
let $x_0\in C_0$, 
let $(t_{i,-1})_{1\leq i\leq m}\in\HH^m$, and iterate
\begin{equation}
\label{e:a8}
\begin{array}{l}
\text{for}\;n=0,1,\ldots\\
\left\lfloor
\begin{array}{l}
\text{for every}\;i\in I_n\\
\left\lfloor
\begin{array}{l}
\text{if}\;\;L_ix_n\notin D_i\\
\left\lfloor
\begin{array}{l}
t_{i,n}=x_n-\gamma\dfrac{\phi_i'\big(d_{D_i}(L_ix_n)\big)}
{d_{D_i}(L_ix_n)}L_i^*\big(L_ix_n-\proj_{D_i}(L_ix_n)\big)\\
\end{array}
\right.\\
\text{else}\\
\left\lfloor
\begin{array}{l}
t_{i,n}=x_n
\end{array}
\right.
\end{array}
\right.\\[5mm]
\text{for every}\;i\in\{1,\ldots,m\}\smallsetminus I_n\\
\left\lfloor
\begin{array}{l}
t_{i,n}=t_{i,n-1}\\
\end{array}
\right.\\[1mm]
x_{n+1}=\proj_{C_0}\Bigg(\Sum_{i=1}^m\omega_it_{i,n}\Bigg).
\end{array}
\right.\\
\end{array}
\end{equation}
Then the following hold:
\begin{enumerate}
\item 
\label{c:8i}
$(x_n)_{n\in\NN}$ converges weakly to a solution to
Problem~\ref{prob:8}.
\item 
\label{c:8ii}
Suppose that one of the following holds:
\begin{enumerate}[label={\rm[\alph*]}]
\setlength{\itemsep}{0pt}
\setcounter{enumii}{4}
\item
\label{c:8e}
Condition~\ref{c:8b} is satisfied with the additional assumptions
that $\phi_j=\mu_j|\cdot|^2/2$ and $D_j$ is compact.
\item
\label{c:8f}
$C_0$ is boundedly compact. 
\end{enumerate}
Then $(x_n)_{n\in\NN}$ converges strongly to a solution to
Problem~\ref{prob:8}.
\end{enumerate}
\end{corollary}
\begin{proof}
We first note that \eqref{e:prob8} is an instance of 
\eqref{e:prob3} with $f_0=\iota_{C_0}$ and 
$(\forall i\in\{1,\ldots,m\})$ $f_i=\phi_i\circ d_{D_i}\circ L_i$. 
Next, we derive from \cite[Example~2.7]{Livre1} that, for every 
$i\in\{1,\ldots,m\}$, $f_i$ is convex and differentiable, and 
that its gradient
\begin{equation}
\label{e:kj80}
\nabla f_i\colon\HH\to\HH\colon x\mapsto
\begin{cases}
\dfrac{\phi_i'\big(d_{D_i}(L_ix)\big)}{d_{D_i}(L_ix)}
L_i^*\big(L_ix-\proj_{D_i}(L_ix)\big),
&\text{if}\;\;L_ix\notin {D_i};\\
0,&\text{if}\;\;L_ix\in {D_i}
\end{cases}
\end{equation}
has Lipschitz constant $\mu_i\|L_i\|^2$. Hence \eqref{e:a8} is an
instance of \eqref{e:a4}. Now, in order to apply
Proposition~\ref{p:3}, let us check that \eqref{e:exi} is satisfied 
under one of assumptions \ref{c:8a}--\ref{c:8d}.

\ref{c:8a}: We have 
$f_0(x)+\sum_{i=1}^m\omega_i f_i(x)\geq
\omega_j(\iota_{C_0}(x)+ f_j(x))\to\pinf$ as
$\|x\|\to\pinf$.

\ref{c:8b}$\Rightarrow$\ref{c:8a}: In view of \ref{c:8d}, we
assume that $C_0$ is unbounded. It follows from Lemma~\ref{l:mi7}
that there exists $\rho\in\RPP$ such that 
$(\forall x\in(\ker L_j)^\bot)$ $\|L_jx\|\geq\rho\|x\|$.
Hence,
\begin{equation}
\label{e:mi1}
(\forall x\in C_0)\quad\|L_jx\|\geq\rho\|x\|.
\end{equation}
Now let $z\in\GG_j$. Then, since $D_j$ is bounded,
$\delta=\diam(D_j)+\|\proj_{D_j}z\|<\pinf$ and 
\begin{equation}
(\forall y\in\GG_j)\quad\|y\|
\leq\|y-\proj_{D_j}y\|+\|\proj_{D_j}y-\proj_{D_j}z\|
+\|\proj_{D_j}z\|\leq d_{D_j}(y)+\delta.
\end{equation}
Consequently, $d_{D_j}(y)\to\pinf$ as $\|y\|\to\pinf$ with
$y\in\GG_j$. Thus, since $\phi_j$ is coercive by
\cite[Proposition~16.23]{Livre1}, we obtain
\begin{equation}
\label{e:mi4}
\phi_j\big(d_{D_j}(y)\big)\to\pinf\quad\text{as}\quad
\|y\|\to\pinf\quad\text{with}\quad y\in\GG_j.
\end{equation}
We deduce from \eqref{e:mi1} and \eqref{e:mi4} that
\begin{equation}
\label{e:mi5}
f_j(x)\to\pinf\quad\text{as}\quad\|x\|\to\pinf\quad
\text{with}\quad x\in C_0.
\end{equation}

\ref{c:8c}$\Rightarrow$\ref{c:8b} and 
\ref{c:8d}$\Rightarrow$\ref{c:8a}: Clear.\\
\noindent
We are now ready to use Proposition~\ref{p:3} to prove the
assertions. 

\ref{c:8i}: Apply Proposition~\ref{p:3}\ref{p:3ii}.

\ref{c:8ii}\ref{c:8e}: Let $x$ be the weak limit in \ref{c:8i} and
set $u_j=\nabla f_j(x)/\mu_j$. Then
Proposition~\ref{p:3}\ref{p:3i} asserts that 
\begin{equation}
\label{e:mi2}
L_j^*\big(L_jx_n-\proj_{D_j}(L_jx_n)\big)\to u_j.
\end{equation}
We also observe that, since $L_j^*\circ L_j$ is weakly continuous
\cite[Lemma~2.41]{Livre1}, we have $L^*_j(L_jx_n)\weakly
L^*_j(L_jx)$. Therefore, \eqref{e:mi2} yields 
\begin{equation}
\label{e:mi3}
L_j^*\big(\proj_{D_j}(L_jx_n)\big)\weakly L^*_j(L_jx)-u_j.
\end{equation}
However, the set $L_j^*(D_j)$ is compact by
\cite[Lemma~1.20]{Livre1} and it contains
$(L_j^*(\proj_{D_j}(L_jx_n)))_{n\in\NN}$. This sequence has 
therefore $L^*_j(L_jx)-u_j$ as its unique strong sequential cluster
point. Thus, $L_j^*(\proj_{D_j}(L_jx_n))\to L^*_j(L_jx)-u_j$ and 
we deduce from \eqref{e:mi2} that 
\begin{equation}
\label{e:mi6}
L^*_j(L_jx_n)\to L^*_j(L_jx).
\end{equation}
On the other hand, for every $n\in\NN$, since $x$ and $x_n$ lie in
$C_0\subset(\ker L_j)^\bot$, we have 
$x_n-x\in(\ker L_j)^\bot=(\ker L_j^*\circ L_j)^\bot$.
Hence, we deduce from \eqref{e:mi6} and Lemma~\ref{l:mi7} that
there exists $\theta\in\RPP$ such that 
\begin{equation}
\label{e:mi9}
\theta\|x_n-x\|\leq\|(L^*_j\circ L_j)(x_n-x)\|\to 0.
\end{equation}
We conclude that $x_n\to x$. 

\ref{c:8ii}\ref{c:8f}: This follows from
Proposition~\ref{p:3}\ref{p:3iiib} since the lower level 
sets of $f_0$ are the compact sets $\{\emp,C_0\}$.
\end{proof}

We conclude by revisiting \eqref{e:cheney59} and recovering a 
classical result on the method of alternating projections.

\begin{example}{\rm\cite[Theorem~4(a)]{Che59a}}
\label{ex:8}
Let $C$ and $D$ be nonempty closed convex subsets of $\HH$ such
that $D$ is compact. Let $x_0\in\HH$ and set $(\forall n\in\NN)$
$x_{n+1}=\proj_{C}(\proj_{D}x_n)$. Then $(x_n)_{n\in\NN}$
converges strongly to a point in $x\in C$ such that 
$x=\proj_C(\proj_Dx)$. 
\end{example}
\begin{proof}
Apply Corollary~\ref{c:8}\ref{c:8ii}\ref{c:8e} with $m=1$, 
$C_0=C$, $\GG_1=\HH$, $L_1=\Id$, $D_1=D$, $\gamma=1$, and
$\mu_1=1$.
\end{proof}

\end{document}